\documentclass[11pt]{article}
\usepackage{latexsym}
\usepackage{graphicx}
\usepackage{latexsym}
\usepackage{graphicx}
\usepackage{amssymb}
\usepackage{amsmath}

\newtheorem{definition}{Definition}

\newtheorem{theorem}{Theorem}
\newtheorem{claim}{Claim}
\newtheorem{example}{Example}

\newenvironment{proof}{{\bf Proof.}}{\hfill $\Box$ \bigskip}
\usepackage{stmaryrd} 

\title{Invariant Manifolds
for Competitive Systems in the Plane}

\author{M. R. S. Kulenovi\'{c} and Orlando Merino \\
Department of Mathematics \\
University of Rhode Island \\
Kingston, Rhode Island 02881-0816 }
\date{\today}

\begin{document}
\maketitle
\thispagestyle{empty}
\begin{abstract}
Let $T$ be a competitive map  on a rectangular region
$\mathcal{R}\subset \mathbb{R}^2$, and assume $T$ is
$C^1$ in a neighborhood of a fixed point
$\overline{\rm x}\in \mathcal{R}$.
The main results of this paper give  conditions on $T$ that guarantee the existence of
an invariant curve emanating from $\overline{\rm x}$
when both eigenvalues of the Jacobian of $T$
at $\overline{\rm x}$ are nonzero and at least one of them
has absolute value less than one,
and establish that $\mathcal{C}$  is an increasing curve
that separates $\mathcal{R}$ into invariant regions.
The results apply to many hyperbolic and nonhyperbolic cases,
and can be effectively used to determine basins of attraction
of fixed points of competitive maps,
or equivalently,
of equilibria of competitive systems of difference equations.
Several applications to
planar systems of difference equations with
non-hyperbolic equilibria are given.
\end{abstract}
\newpage
\section{Introduction and main results}
\label{sec: Introduction}
The following system of difference equations is analyzed in \cite{CK} for $a>1$:
\begin{equation}
\label{equation:compet2dclarkulb=1}
\left\{
\begin{array}{rcl}
x_{n+1} &=& \displaystyle{\frac{x_n}{a + y_n}} \\  \\
y_{n+1} &=& \displaystyle{\frac{y_n}{1 + x_n}}
\end{array}
\right.
\, ,
\quad n=0,1,2,\ldots,
\quad x_{0},\, y_0 \geq 0\, .
\end{equation}
It is shown there that every point $(0, \bar{y})$ on the positive section of the $y$-axis is a
non-hyperbolic equilibrium point with
eigenvalues of the Jacobian of the associated map at the point,
or characteristic values,  given by
$\lambda_1=1$ and $\lambda_2=\frac{1}{a+\bar{y}}$.
It is also shown in \cite{CK} that for each $\overline{y}>0$, equation
(\ref{equation:compet2dclarkulb=1}) possesses solutions which converge to
$(0, \bar{y})$.
Parts of the basins of attraction of the equilibrium points $(0, \bar{y})$
were found in \cite{CK},
and the global behavior of solutions of system (\ref{equation:compet2dclarkulb=1}) was
characterized completely, with the exception of determining the basin of attraction of each equilibrium point.
Extensive simulations conducted by the authors of \cite{CK}
suggest that the basin of attraction of each equilibrium point on the $y$-axis
 is   the graph of a continuous increasing function on $[0,\infty)$, but
they  were not able to prove this fact.
A similar phenomenon has been observed in \cite{CaKuLaMe}
in several special cases of {\em competitive systems} of the form
\begin{equation}
\label{system1}
\left\{
\begin{array}{l}
x_{n+1}=\displaystyle{\frac{\alpha_{1}+\beta_{1} x_{n}+\gamma_{1} y_{n}}{A_{1}+B_{1} x_{n}+C_{1}y_{n}}}\\
\\
y_{n+1}=\displaystyle{\frac{\alpha_{2}+\beta_{2} x_{n}+\gamma_{2} y_{n}}{A_{2}+B_{2} x_{n}+C_{2}y_{n}}}\\
\end{array}
\right., \;\; n=0,1,\ldots\, ,
\end{equation}
with  nonnegative parameters $\alpha_{1}$, $\beta_{1}$,
$\gamma_{1}$, $A_{1}$, $B_{1}$, $C_{1}$, $\alpha_{2}$,
$\beta_{2}$, $\gamma_{2}$, $A_{2}$, $B_{2}$, $C_{2}$ and with
arbitrary nonnegative initial conditions $x_{0}$, $y_{0}$ such
that the denominators are always positive. See Open Problems 1-3 in \cite{CaKuLaMe}.

A first order system of difference equations
\begin{equation}
\label{equation:compet2d}
\left\{
\begin{array}{rcl}
x_{n+1} &=& f(x_n,y_n) \\
y_{n+1} &=& g(x_n,y_n)
\end{array}
\right.
\, ,
\quad n=0,1,2,\ldots,
\quad (x_{-1},x_0)\in \mathcal{R}\, ,
\end{equation}
where $\mathcal{R} \subset \mathbb{R}^2$, $(f, g): \mathcal{R} \rightarrow  \mathcal{R}$,
$f$, $g$ are continuous functions
is  {\it competitive} if
$f(x,y)$ is non-decreasing in $x$ and non-increasing in $y$, and $g(x, y)$ is non-increasing in $x$ and non-decreasing in $y$.
If both $f$ and $g$ are nondecreasing in $x$ and $y$, the system
(\ref{equation:compet2d}) is {\it cooperative}.
Competitive and cooperative maps are defined similarly.
{\it Strongly  competitive} systems of difference equations or maps are those for which the functions
$f$ and $g$ are coordinate-wise stricly monotone.
Competitive and cooperative systems of difference equations of the form (\ref{equation:compet2d})
have been studied by many authors
\cite{BM, BKK1, CaKuLaMe, CK, CKS, CAKS, CLCH, Dancer and Hess, deMS, FY1, FY2, FY3, HP, HS1, KUM1, KUM2, KN4, PT1, S1986JDE, S1986SIAM, S2-1986SIAM,  Sm1}
 and others.

A classical result of Poincar\'e, Hadamard, and Sternberg
(see Lemma 5.1  in page 234 of \cite{Hartman} and the Notes section  in page 271 therein)
gives conditions for the existence of a {\it local} smooth curve through
the fixed point of a smooth map on the plane when
the characteristic values are real and distinct and such that one of them is smaller than $1$.
The  well known Stable Manifold Theorem   has local character and
applies to hyperbolic cases of very general maps in the plane, for example see \cite{ASY}.
The local stable manifold of a diffeomorphism is always a ``nice'' curve, but  in general
the corresponding global stable manifold may be a very complex set,
for example it is a strange attractor  in the case of Henon's system,
see \cite{ASY} and references therein. See also \cite{KUM1}.
H. L. Smith  \cite{S1986SIAM}
obtained results for fixed points of smooth maps on Banach space
and showed that under certain conditions,
if the Frechet derivative of the map at the fixed point is an eigenvalue larger than one,
then there is an invariant curve emanating from the fixed point, which Smith termed the
``most unstable manifold''.  Smith also showed that if the map has certain monotonicity
conditions, then the curve is monotone.
See also \cite{SZ,CKS}.

The first result of this article
gives conditions for the existence of a global invariant curve through
a fixed point (hyperbolic or not) of a competitive map that is differentiable
in a neighborhood of the fixed point, when at least one of two nonzero eigenvalues
of the Jacobian of the map at the fixed point has absolute
value less than one.
Proofs for all theorems in this section
will be given in Section \ref{sec: Proof}.
A region $\mathcal{R}\subset \mathbb{R}^2$ is {\em rectangular}
if it is the cartesian product of two intervals in $\mathbb{R}$.
By ${\rm int}\,\mathcal{A}$ we denote the interior of a set $\mathcal{A}$.
\begin{theorem}
\label{th: invariant curve}
Let  $T$ be a competitive
map on a rectangular region $\mathcal{R}\subset \mathbb{R}^2$.
Let  $\overline{\rm x}\in \mathcal{R}$ be a fixed point of $T$
such that $\Delta:=\mathcal{R}\cap {\rm int}\,( \mathcal{Q}_{1}(\overline{\rm x}) \cup \mathcal{Q}_3(\overline{\rm x}))$
is nonempty (i.e., $\overline{\rm x}$ is not the NW or SE vertex of $\mathcal{R})$,
and $T$ is strongly competitive on $\Delta$.
Suppose that the following statements are true.
\begin{itemize}
\item[\rm a.]
The map $T$ has a  $C^1$ extension to a  neighborhood of $\overline{\rm x}$.
\item[\rm b.]
The Jacobian $J_T(\overline{\rm x}) $ of $T$ at $\overline{\rm x}$ has real eigenvalues
$\lambda$, $\mu$ such that  $0< |\lambda|< \mu$, where $|\lambda|<1$,
and the eigenspace $E^\lambda$
associated with $\lambda$ is not a coordinate axis.
\end{itemize}
Then there exists a   curve $\mathcal{C} \subset \mathcal{R}$
through $\overline{{\rm x}}$ that is invariant and
 a subset
of the basin of attraction of $\overline{{\rm x}}$, such that
$\mathcal{C}$ is tangential to the eigenspace $E^\lambda$ at $\overline{{\rm x}}$,
and $\mathcal{C}$ is the graph of a strictly increasing continuous
function of the first coordinate on an interval.
Any endpoints of  $\mathcal{C}$ in the interior of $\mathcal{R}$
are either fixed points or minimal period-two points.  In the latter case,
the set of  endpoints of $\mathcal{C}$ is a minimal period-two orbit of $T$.
\end{theorem}
We shall see in Theorem \ref{th: basins}
and in the Section \ref{sec: Applications} examples that
the situation where
 the endpoints of $\mathcal{C}$
are boundary points of $\mathcal{R}$
is of interest.
The following result gives a sufficient condition for this case.
\begin{theorem}
\label{th: suff cond for boundary endpoints}
\label{cor: cor to invariant curve}
For the curve $\mathcal{C}$ of Theorem \ref{th: invariant curve}
to have endpoints in $\partial \mathcal{R}$,
it is sufficient that at least one
of the following conditions is satisfied.
\begin{itemize}
\item[\rm i.]
The map $T$ has no fixed points nor periodic points of minimal period-two in
$\Delta$.
\item[\rm ii.]
The map $T$ has no fixed points in $\Delta$,
$\det \, J_T(\overline{\rm x}) > 0$, and
$T({\rm x}) = \overline{\rm x}$ has no solutions ${\rm x} \in \Delta$.
\item[\rm iii.]
The map $T$ has no points of minimal period-two in $\Delta$,
$\det \, J_T(\overline{\rm x}) < 0$, and
$T({\rm x}) = \overline{\rm x}$ has no solutions ${\rm x} \in \Delta$.
\end{itemize}
\end{theorem}
In many cases one can expect the curve $\mathcal{C}$ to be smooth.
\begin{theorem}
\label{th: smooth manifold}
Under the hypotheses of Theorem \ref{th: invariant curve}, suppose there exists
a neighborhood $\mathcal{U}$ of $\overline{\rm x}$ in $\mathbb{R}^2$
such that $T$ is of class $C^k$ on $\mathcal{U}\cup \Delta$ for some $k\geq 1$,
and that the Jacobian of $T$ at each ${\rm x}\in \Delta$
is invertible.
Then the curve $\mathcal{C}$  in the conclusion of
Theorem \ref{th: invariant curve} is of class $C^k$.
\end{theorem}
In applications it is common to have rectangular domains $\mathcal{R}$ for competitive maps.
If  a competitive map exist has several fixed points,
often the domain of the map may be split into rectangular invariant subsets such that
 Theorem \ref{th: invariant curve} could be applied to the restriction of the map
 to one or more subsets.
For maps that are strongly competitive near the fixed point,
hypothesis b. of Theorem \ref{th: invariant curve}
reduces just to $|\lambda | <1$.
This  follows  from a change
of variables  \cite{Sm1} that allows the Perron-Frobenius Theorem to be applied to give that
at any point,  the Jacobian of a strongly competitive map has two real
and distinct eigenvalues, the larger one in absolute value being
positive, and  that
corresponding eigenvectors may be chosen to point in
the direction of the second and first quadrant, respectively.
Also, one can show that in such case
no associated eigenvector is aligned with a coordinate axis.

Lemma 5.1  in page 234 of \cite{Hartman} is a local version of Theorem \ref{th: invariant curve}
for (not necessarily competitive) $C^1$ planar maps, and it
gives the existence of a $C^1$ local curve  $\mathcal{C}$ as
in Theorem \ref{th: invariant curve}.
%
%
In the case when the map is a diffeomorphism, H. L. Smith's results give the conclusions of
Theorem \ref{th: invariant curve}, see Remark 5 in \cite{S1986SIAM}.
The proof of a result analogous to
Theorem \ref{th: invariant curve} was given in Theorem 5 of
\cite{KUM1}
in the hyperbolic case when the equilibrium is a saddle point,
but a key feature of Theorem  \ref{th: invariant curve}
 is that the equilibrium may be {\it non-hyperbolic}.
 Theorem \ref{th: invariant curve} refines and extends Theorem 5
from \cite{KUM1} in that it only requires smoothness of the map
in a neighborhood of the fixed point, it relaxes the hypothesis
that the fixed point is a saddle, and it removes other hypotheses.

The next result is useful
for determining basins of attraction of fixed points
of competitive maps.
If ${\rm x} \in \mathbb{R}^2$, we denote with $\mathcal{Q}_{\ell}({\rm x})$,
$\ell \in \{1,2,3,4\}$,
the  four quadrants in ${\mathbb{R}}^2$
relative to ${\rm x}$, i.e.,
$\mathcal{Q}_{1}(x,y) =
\{\ (u, v)\in  {\mathbb{R}}^2\ :\ u \geq x,\ v \geq y\ \}$,
$\mathcal{Q}_{2}(x,y) = \{\ (u, v)\in  {\mathbb{R}}^2\  :\ x \geq u,\ v \geq y\ \}$, and so on.
Define the {\it South-East} partial order $\preceq_{se}$ on $\mathbb{R}^2$ by
$(x,y) \preceq_{se} (s,t)$ if and only if $x \leq s$ and $y \geq t$.
For $\mathcal{A}\subset \mathbb{R}^2$ and ${\rm x} \in \mathbb{R}^2$, define
the {\it distance from ${\rm x}$ to $\mathcal{A}$} as
${\rm dist}({\rm x},\mathcal{A}):=\inf \,\{\|{\rm x}-{\rm y}\|\,:{\rm y} \in \mathcal{A}\}$.
\begin{theorem}
\label{th: basins}
{\rm (A)} Asume the hypotheses of Theorem \ref{th: invariant curve},
and let $\mathcal{C}$ be the curve whose existence is guaranteed by Theorem \ref{th: invariant curve}.
If the endpoints of $\mathcal{C}$ belong to $\partial \mathcal{R}$,
then
$\mathcal{C}$ separates $\mathcal{R}$ into two connected components,
namely
\begin{equation}
\mathcal{W}_- := \{ {\rm x} \in \mathcal{R}\setminus \mathcal{C} :
\exists {\rm y} \in \mathcal{C} \mbox{ with } {\rm x} \preceq_{se} {\rm y} \}
\quad \mbox{\rm and} \quad
\mathcal{W}_+ := \{ {\rm x} \in \mathcal{R}\setminus \mathcal{C} :
\exists {\rm y} \in \mathcal{C} \mbox{ with } {\rm y} \preceq_{se} {\rm x} \}\, ,
\end{equation}
such that the following statements are true.
\begin{itemize}
\item[\rm (i)]  $\mathcal{W}_-$ is invariant, and
$\mbox{\rm dist}(T^n({\rm x}),\mathcal{Q}_2(\overline{{\rm x}}))\rightarrow 0$
as $n \rightarrow \infty$
for every ${\rm x} \in \mathcal{W}_-$.
\item[\rm (ii)]  $\mathcal{W}_+$ is invariant, and
$\mbox{\rm dist}(T^n({\rm x}),\mathcal{Q}_4(\overline{{\rm x}}))\rightarrow 0$
as $n \rightarrow \infty$ for every ${\rm x} \in \mathcal{W}_+$.
\end{itemize}
{\rm (B)} If, in addition to the hypotheses of part {\rm (A)},  $\overline{{\rm x}}$ is an interior point of $\mathcal{R}$
and $T$ is $C^2$ and strongly competitive in a neighborhood of $\overline{{\rm x}}$,
then $T$ has no periodic points in the boundary of
$\mathcal{Q}_1(\overline{{\rm x}})\cup \mathcal{Q}_3(\overline{{\rm x}})$
except for $\overline{{\rm x}}$, and the following statements are true.
\begin{itemize}
\item[\rm (iii)]
For every ${\rm x} \in \mathcal{W}_-$ there exists $n_0 \in \mathbb{N}$
such that $T^n({\rm x}) \in {\rm int}\,\mathcal{Q}_2(\overline{{\rm x}})$ for $n\geq n_0$.
\item[\rm (iv)]
For every ${\rm x} \in  \mathcal{W}_+$ there exists $n_0 \in \mathbb{N}$
such that $T^n({\rm x}) \in {\rm int}\, \mathcal{Q}_4(\overline{{\rm x}})$ for $n\geq n_0$.
\end{itemize}
\end{theorem}
Basins of attraction of {\it period-two solutions}
or period-two orbits  of certain systems or maps can be
effectively treated with Theorems \ref{th: invariant curve} and  \ref{th: basins}.
See \cite{KUM1, KUM2, KN4} for the hyperbolic case,
for the non-hyperbolic case see
Example \ref{ex: period two} in Section \ref{sec: Applications},
and reference \cite{BKK1}.

If $T$ is a map on a set $\mathcal{R}$ and if
$\overline{\rm x}$ is a fixed point of $T$, the
{\it stable set} $\mathcal{W}^s(\overline{\rm x})$
of   $\overline{\rm x}$ is the set
$\{ x \in \mathcal{R}: T^n({\rm x}) \rightarrow \overline{\rm x} \}$
and {\it unstable set} $\mathcal{W}^u(\overline{\rm x})$
of   $\overline{\rm x}$ is the set
$$
\left\{ \ {\rm x} \in \mathcal{R}:
\mbox{there exists } \{{\rm x}_n\}_{n=-\infty}^{0} \subset \mathcal{R} \mbox{ s.t. }
T({\rm x}_n)={\rm x}_{n+1} ,\   {\rm x}_0 = {\rm x}, \mbox{ and }
\lim_{n\rightarrow -\infty} {\rm x}_n = \overline{\rm x}
  \ \right\}
$$
When $T$ is non-invertible,
the set $\mathcal{W}^s(\overline{\rm x})$ may not be connected
and made up of infinitely many curves,
or $\mathcal{W}^u(\overline{\rm x})$ may not be a manifold.
The following result gives a description of
the  stable and unstable sets of a saddle point of a
competitive map.
If the map is a diffeomorphism on $\mathcal{R}$, the  sets
$\mathcal{W}^s(\overline{\rm x})$ and $\mathcal{W}^u(\overline{\rm x})$
are the stable and unstable manifolds of $\overline{x}$.
\begin{theorem}
\label{th: saddle point}
In addition to the hypotheses of part (B) of Theorem \ref{th: basins},
suppose that $\mu>1$ and that the eigenspace $E^\mu$ associated with $\mu$
is not a coordinate axis.
If the curve $\mathcal{C}$ of Theorem \ref{th: invariant curve}
has endpoints in $\partial \mathcal{R}$, then $\mathcal{C}$
is the stable set
$\mathcal{W}^s(\overline{\rm x})$ of   $\overline{\rm x}$, and
the unstable set $\mathcal{W}^u(\overline{\rm x})$ of  $\overline{x}$
is a curve in $\mathcal{R}$ that is tangential to $E^\mu$ at
$\overline{\rm x}$ and such that it is the graph of a strictly
decreasing function of the first coordinate on an interval.
Any endpoints of $\mathcal{W}^u(\overline{\rm x})$
in   $\mathcal{R}$ are fixed points of $T$.
\end{theorem}

The following result gives  information on {\it local} dynamics
near a fixed point of a map when there exists a characteristic vector
whose coordinates have negative product
and such that
the associated eigenvalue is hyperbolic.
A point $(x,y)$ is a {\it subsolution} if $T(x,y) \preceq_{se} (x,y)$, and $(x,y)$ is a
{\it supersolution} if $(x,y) \preceq_{se} T(x,y)$.
An {\it order interval}
$\llbracket (a,b),(c,d)\rrbracket$
is the cartesian product of the
two compact intervals $[a,c]$ and $[b,d]$.
\begin{theorem}
\label{th: 2 and 4 all cases}
Let $T$ be a competitive map on a rectangular set $ \mathcal{R}\subset \mathbb{R}^2$
with an isolated fixed point  $\overline{{\rm x}} \in \mathcal{R}$ such that
$ \mathcal{R}\cap \mbox{\rm int } (\mathcal{Q}_2(\overline{{\rm x}})\cup \mathcal{Q}_4(\overline{{\rm x}})) \neq \emptyset$.
Suppose  $T$ has a $C^1$ extension to a neighborhood of $\overline{\rm x}$.
Let ${\rm v} =({\rm v}^{(1)},{\rm v}^{(2)}) \in \mathbb{R}^2$ be an eigenvector of the Jacobian of $T$ at $\overline{{\rm x}}$,
with associated eigenvalue $\mu \in \mathbb{R}$.
If ${\rm v}^{(1)} {\rm v}^{(2)} < 0$,
then there exists an order interval $\mathcal{I}$ which is also a relative neighborhood of
$\overline{{\rm x}}$ such that
for every relative neighborhood $\mathcal{U} \subset \mathcal{I}$ of $\overline{{\rm x}}$
the following statements are true.
\begin{itemize}
\item[\rm i.]
If $\mu >1$, then
$\mathcal{U} \cap\,{\rm int}\,\mathcal{Q}_2(\overline{{\rm x}})$
contains a subsolution and
$\mathcal{U} \cap\,{\rm int}\,\mathcal{Q}_4(\overline{{\rm x}})$
contains a supersolution.
In this case
for every ${\rm x} \in \mathcal{I} \cap \mbox{\rm int} (\, \mathcal{Q}_2(\overline{{\rm x}})\cup \mathcal{Q}_4(\overline{{\rm x}})\,)$
 there exists $N$ such that
$T^n({\rm x}) \not \in \mathcal{I}$ for $n \geq N$.
\item[\rm ii.]
If $\mu <1$, then
$\mathcal{U} \cap\,{\rm int}\,\mathcal{Q}_2(\overline{{\rm x}})$
contains a supersolution and
$\mathcal{U} \cap\,{\rm int}\,\mathcal{Q}_4(\overline{{\rm x}})$
contains a subsolution.
In this case $T^n({\rm x}) \rightarrow \overline{{\rm x}}$ for every
${\rm x} \in \mathcal{I}$.
\end{itemize}
\end{theorem}
In the non-hyperbolic case, we have the following result.
\begin{theorem}
\label{th: 2 and 4 all cases part 2}
Assume that the hypotheses of Theorem \ref{th: 2 and 4 all cases} hold,
that $T$ is real analytic at $\overline{{\rm x}}$, and that $\mu = 1$.
Let $c_j$, $d_j$, $j=2,3,\ldots$
be defined by the Taylor series
\begin{equation}
\label{eq: taylor}
T(\overline{{\rm x}}+t\, {\rm v}) = \overline{{\rm x}} + {\rm v} \,t+ (c_2,d_2) \,t^2 + \cdots + (c_n,d_n) \,t^n + \cdots \, .
\end{equation}
Suppose that there exists an index $\ell \geq 2$ such that
$(c_\ell,d_\ell)\neq (0,0)$ and
$(c_j, d_j) = (0,0)\quad  \mbox{\rm for } j < \ell$.
If either
$$
{\rm (a)}\, c_\ell\, d_\ell < 0\, , \  \mbox{\rm or} \
{\rm (b)}\, c_\ell \neq 0\ \mbox{\rm and } \, T(\overline{{\rm x}}+t\,{\rm v})^{(2)} \mbox{ \rm is affine in $t$},
\  \mbox{\rm or} \
{\rm (c)}\, d_\ell \neq 0\ \mbox{\rm and } \, T(\overline{{\rm x}}+t\,{\rm v})^{(1)} \mbox{ \rm is affine in $t$},
$$
then
there exists an order interval $\mathcal{I}$ which is also a relative neighborhood of
$\overline{{\rm x}}$ such that
for every relative neighborhood $\mathcal{U} \subset \mathcal{I}$ of $\overline{{\rm x}}$
the following statements are true.
\begin{itemize}
\item[\rm i.]
If $\ell$ is odd and $(c_\ell,d_\ell) \preceq_{se} (0,0)$, then
$\mathcal{U} \cap \,{\rm int}\,\mathcal{Q}_4(\overline{{\rm x}})$ contains a supersolution and
$\mathcal{U} \cap \,{\rm int}\,\mathcal{Q}_2(\overline{{\rm x}})$ contains a subsolution.
In this case,
for every ${\rm x} \in \mathcal{I} \cap \mbox{\rm int} (\, \mathcal{Q}_2(\overline{{\rm x}})\cup \mathcal{Q}_4(\overline{{\rm x}})\,)$
 there exists $N$ such that
$T^n({\rm x}) \not \in \mathcal{I}$ for $n \geq N$.
\item[\rm ii.]
If $\ell$ is odd and $(0,0) \preceq_{se} (c_\ell,d_\ell) $,  then
$\mathcal{U} \cap \,{\rm int}\,\mathcal{Q}_4(\overline{{\rm x}})$ contains a subsolution and
$\mathcal{U} \cap \,{\rm int}\,\mathcal{Q}_2(\overline{{\rm x}})$ contains a supersolution.
In this case,
$T^n({\rm x}) \rightarrow \overline{{\rm x}}$ for every ${\rm x} \in \mathcal{I}$.
\item[\rm iii.]
If $\ell$ is even and $(c_\ell,d_\ell) \preceq_{se} (0,0)$, then
$\mathcal{U} \cap \,{\rm int}\,\mathcal{Q}_4(\overline{{\rm x}})$ contains a subsolution and
$\mathcal{U} \cap \,{\rm int}\,\mathcal{Q}_2(\overline{{\rm x}})$ contains a subsolution.
In this case,
$T^n({\rm x}) \rightarrow \overline{{\rm x}}$ for every ${\rm x} \in \mathcal{I} \cap \mathcal{Q}_4(\overline{{\rm x}})$, and
for every ${\rm x} \in \mathcal{I} \cap \mbox{\rm int} (\, \mathcal{Q}_2(\overline{{\rm x}})\,)$
 there exists $N$ such that
$T^n({\rm x}) \not \in \mathcal{I}$ for $n \geq N$.
\item[\rm iv.]
If $\ell$ is even and $(0,0) \preceq_{se} (c_\ell,d_\ell)$, then
$\mathcal{U} \cap \,{\rm int}\,\mathcal{Q}_2(\overline{{\rm x}})$ contains a supersolution and
$\mathcal{U} \cap \,{\rm int}\,\mathcal{Q}_4(\overline{{\rm x}})$ contains a supersolution.
In this case,
$T^n({\rm x}) \rightarrow \overline{{\rm x}}$ for every ${\rm x} \in \mathcal{I} \cap \mathcal{Q}_2(\overline{{\rm x}})$, and
for every ${\rm x} \in \mathcal{I} \cap \mbox{\rm int} (\,  \mathcal{Q}_4(\overline{{\rm x}})\,)$
 there exists $N$ such that
$T^n({\rm x}) \not \in \mathcal{I}$ for $n \geq N$.
\end{itemize}
\end{theorem}

The rest of this paper is organized as follows.
In Section \ref{sec: Competitive} some definitions and background on
competitive maps and systems of difference equations is given.
In Section \ref{sec: Applications} we present applications of the main results
to several classes of difference equations that depend on parameters.
In Example 1 we study system (\ref{equation:compet2dclarkulb=1}),
which has a continuum of non-hyperbolic equilibria
along a vertical line.  Theorem \ref{th: invariant curve} is used to establish
that
the stable set of each equilibrium point is an increasing curve,
and that
the limiting equilibrium of each solution is a continuous function of the
initial point.
Example 2 completes the analysis of a system of difference equation
that was studied in  \cite{CLCH, KUM1}  for the hyperbolic equilibria case.
Here we consider the case of non-hyperbolic equilibria,
for which there exists a line segment of such equilibria,
and we show that each of them has a global invariant set
given by an  increasing curve.
In Example 3, we apply our results to a difference equation
and obtain  stable sets for
each of the period-two points, which are non-hyperbolic and
consist of all points on a hyperbola.
Example 4 exhibits a system of difference equations with a unique equilibrium
which is of non-hyperbolic type and semi-stable.
The equilibrium is of
oscillatory type.
Theorems \ref{th: invariant curve},
\ref{th: basins} and \ref{th: 2 and 4 all cases part 2}
are used to establish global behavior of solutions.
For this we also used the competitive character of the system, as well as
information on eigenvectors associated to the characteristic values at the
equilibrium.
Example 5 is about a system with a semi-stable non-hyperbolic interior
equilibrium.
Only qualitative information is assumed about this system (two equilibria
exist),
yet this is all is needed to characterize the basins of attraction
of the two equilibria.  Our results here expand and complete the analysis
given in \cite{BM}.
In Section \ref{sec: Proof} the proofs of
Theorems \ref{th: invariant curve} -- \ref{th: 2 and 4 all cases part 2}  are
presented.
\section{Competitive and cooperative systems and maps}
\label{sec: Competitive}
We shall restrict our discussion to competitive systems,
since if system (\ref{equation:compet2d}) is cooperative,
a simple change of variables  yields a competitive system, see \cite{Sm1}.
Also, applications require the region $\mathcal{R}$ to be the cartesian product of
intervals in $\mathbb{R}$, which we shall assume in our main result.
\bigskip

The most natural way to study properties of competitive and cooperative
systems (\ref{equation:compet2d}) is to consider the corresponding maps
$T:\mathcal{R}\rightarrow \mathcal{R}$ where $T(x,y) = (f(x,y),g(x,y))$, $(x,y) \in \mathcal{R}$,
since such maps are {\it order preserving}
or {\it monotone}, i.e.,
$T(x^{(1)},y^{(1)}) \preceq T(x^{(2)},y^{(2)})$
whenever $(x^{(1)},y^{(1)}) \preceq  (x^{(2)},y^{(2)})$,
where $\preceq$ is a suitable partial order in $\mathbb{R}^2$.
Consider the ``North-East'' and  ``South-East'' partial orders in $\mathbb{R}^2$
given by
$$
\mbox{
$(x^{(1)},y^{(1)}) \preceq_{ne} (x^{(2)},y^{(2)})$ if and only if
$x^{(1)} \leq x^{(2)}\ \mbox{and} \ y^{(1)}\leq y^{(2)}$
}\, ,
$$
and
$$
\mbox{
$(x^{(1)},y^{(1)}) \preceq_{se} (x^{(2)},y^{(2)})$ if and only if $x^{(1)} \leq x^{(2)}\ \mbox{and} \ y^{(1)}\geq y^{(2)}$
}\, .
$$
We shall use the notation $\overline{0}$ to represent the origin $(0,0)$ in $\mathbb{R}^2$.
The first quadrant $\mathcal{Q}_1(\overline{0})=\{(x,y): x \geq 0, \ y\geq 0\}$ is the nonnegative cone
associated to $\preceq_{ne}$, and the fourth quadrant $\mathcal{Q}_4(\overline{0})=\{(x,y): x \geq 0, \ y\leq 0\}$
is the nonnegative cone associated to $\preceq_{se}$.

From the definition of cooperative and competitive systems
 one can see that maps of cooperative (respectively, competitive)
systems are monotone with respect to $\preceq_{ne}$
(resp. $\preceq_{se}$).
Note that strongly competitive maps $T$ satisfy the relation
$
(x,y) \preceq_{se} (w,z) \implies T(w,z)-T(x,y) \in {\rm int}\, \mathcal{Q}_4(\overline{0})
$.

Consider $\mathbb{R}^2$ equipped with a partial order $\preceq$
equal to either $\preceq_{ne}$ or $\preceq_{se}$, that is,
the nonnegative cone is $P=\mathcal{Q}_1(\overline{0})$ or $P=\mathcal{Q}_4(\overline{0})$.
We say that ${\rm x,y} \in \mathbb{R}^2$ are {\it comparable} in the order $\preceq$ if
either ${\rm x} \preceq {\rm y}$ or ${\rm y} \preceq {\rm x}$.
For ${\rm x}, {\rm y} \in \mathbb{R}^2$ such that ${\rm x} \prec {\rm y}$,
the {\it order interval} $\llbracket{\rm x},{\rm y}\rrbracket$ is the set
of all $\rm z$ such that ${\rm x}\preceq {\rm z} \preceq {\rm y}$.
A set $\mathcal{A}$ is said to be {\em linearly ordered}  if
$\preceq$ is a total order on $\mathcal{A}$.

A {\it map} $T$ on a set $\mathcal{B} \subset \mathbb{R}^2$ is a continuous function
$T:\mathcal{B}\rightarrow \mathcal{B}$.
A set $\mathcal{A} \subset \mathcal{B}$ is {\it invariant} for the map $T$ if
$T(\mathcal{A}) \subset \mathcal{A}$.
The {\it omega-limit set} of a point ${\rm z}\in \mathcal{A}$ is the set
$\omega({\rm z}) =\{ {\rm w} \in \mathcal{R}:\exists n_k \rightarrow \infty \mbox{ such that } T^{n_k}({\rm z}) \rightarrow {\rm w}\}$.
A point ${\rm x} \in \mathcal{B}$ is a {\it fixed point} of $T$ if
$T({\rm x}) = {\rm x}$, and a {\it minimal period-two point} if $T^2({\rm x})={\rm x}$
and $T({\rm x}) \neq {\rm x}$.
The {\it orbit of ${\rm x} \in \mathcal{B}$} is the sequence
$\{ T^\ell({\rm x})\}_{\ell=0}^\infty$.
A {\it minimal period-two orbit} is an orbit $\{ {\rm x}_\ell\}_{\ell=0}^\infty$
for which ${\rm x}_{0} \neq {\rm x}_1$ and ${\rm x}_{0} = {\rm x}_2$.
The {\it basin of attraction} of a fixed point ${\rm x}$ is the set of all
${\rm y}$ such that $T^n({\rm y})\rightarrow {\rm x}$.
A fixed point ${\rm x}$ is a {\it global attractor} of a set $\mathcal{A}$
if $\mathcal{A}$ is a subset of the basin of attraction of ${\rm x}$.

The map is {\it smooth} on $\mathcal{B}$ if the interior of $\mathcal{B}$ is nonempty
and if $T$ is continuously differentiable on the interior of $\mathcal{B}$.
If $T$ is differentiable, a sufficient condition for $T$ to be
strongly competitive is that the Jacobian matrix of $T$ at any ${\rm x} \in \mathcal{B}$
has the sign configuration
$$\left(\begin{array}{cc}+ & - \\ - & +\end{array}\right).
$$
System (\ref{equation:compet2d}) has an associated map
$T = (f,g)$ defined on the set $\mathcal{R}$.
For additional definitions and results  (e.g., repellor,  hyperbolic fixed points,
stability, asymptotic stability, stable and unstable sets and manifolds)
see \cite{Robinson} for maps,
\cite{HS1} and \cite{Sm1} for competitive maps, and \cite{KUM, KUM1}
for difference equations.

The next two theorems gives sufficient conditions for a
competitive system to have solutions that are component-wise
eventually monotonic.

Following Smith \cite{Sm1}, we introduce
\begin{definition}
\label{def: O+}
A competitive map $T:\mathcal{R}\rightarrow \mathcal{R}$, $\mathcal{R}\subset \mathbb{R}^2$,
 is said to satisfy condition
($O+$)  if for every $x$, $y$ in $\mathcal{R}$,
$T(x) \preceq_{ne} T(y)$ implies $x \preceq_{ne} y$.
The map $T$ is said to satisfy condition
($O-$)  if for every $x$, $y$ in $\mathcal{R}$,
$T(x) \preceq_{ne} T(y)$ implies $y \preceq_{ne} x$.
\end{definition}
The following theorem was proved by DeMottoni-Schiaffino
\cite{deMS} for the Poincar\'e map of a periodic
competitive Lotka-Volterra system of differential equations.
Smith generalized the proof to competitive and cooperative maps
 \cite{S1986JDE, S1986SIAM}.

\begin{theorem}
\label{thm: demottoni}
If $T:\mathcal{R}\rightarrow \mathcal{R}$, $\mathcal{R}\subset \mathbb{R}^2$,
 is a competitive map for which ($O+$) holds then
for all $x \in \mathcal{R}$, $\{T^n(x)\}$ is eventually componentwise monotone.
If the orbit of $x$ has compact closure, then it converges to
a fixed point of $T$.
If instead ($O-$) holds, then for all $x \in \mathcal{R}$,
$\{T^{2n}\}$ is eventually componentwise monotone.
If the orbit of $x$ has compact closure in $\mathcal{R}$, then
its omega limit set is either a period-two orbit or a fixed
point.
\end{theorem}
The following result is Lemma 4.3 from \cite{Sm1} specialized to
smooth maps on planar rectangular regions.
If $T$ is a map which is differentiable at a point $\rm x$,
by $J_T({\rm x})$ we denote the Jacobian matrix of $T$ at $x$.
\begin{theorem}
\label{thm: suff cond O+}
Let $T$ be a $C^1$ competitive map on a rectangular region $\mathcal{R}$.
If $T$ is injective and
 $\mbox{\rm det} J_T(x) > 0$  for all $x \in \mathcal{R}$
 then $T$ satisfies ($O+$).
 If $T$ is injective and
 $\mbox{\rm det} J_T(x) < 0$ for all $x \in \mathcal{R}$
 then $T$ satisfies ($O-$).
\end{theorem}
H. L. Smith performed a systematic study of competitive and cooperative maps
and in particular introduced
invariant manifolds techniques in his analysis   \cite{S1986JDE, S1986SIAM,  S2-1986SIAM}
 with some results valid for maps on  $n$-dimensional space.
Smith restricted attention mostly to competitive maps $T$ that satisfy
additional constraints.
In particular, $T$ is required to be
a diffeomorphism of
a neighborhood of $\mathbb{R}^n_+$
that satisfies either ($O+$) or ($O-$),
(this is the case if $T$ is orientation-preserving or orientation-reversing), and
that the coordinate semiaxes are invariant under $T$.
The latter requirement is a common feature of many population dynamics applications,
where a point on a positive semiaxis is interpreted as one of the populations
having no individuals, and thus the corresponding orbit terms having
the same characteristic.
For such class of maps (as well as for cooperative maps satisfying similar hypotheses)
Smith obtained results on invariant manifolds
passing through fixed points and a fairly complete description of the
phase-portrait when $n=2$, especially for
those cases having a unique fixed point on each of the
open positive semiaxes.

\section{Applications}
\label{sec: Applications}
In this section we present several applications of our main results.
The examples are of non-hyperbolic type.
The hyperbolic case is well known and has been treated in
\cite{CKS,CAKS,KUM1,KUM2,KN4}.
\begin{example}
A system with a continuum of non-hyperbolic equilibria along a vertical line.
\rm
Consider system (\ref{equation:compet2dclarkulb=1}) with $a>1$.
The map of the system is
$$
T(x,y) = \left(\frac{x}{a+y},\frac{y}{1+x}\right),\quad (x,y) \in [0,\infty)^2\, .
$$
The fixed points of $T$ have the form $(0, \bar{y})$, with $\overline{y} \geq 0$.
The map is smooth and strongly competitive on $[0,\infty)^2$.
One eigenvalue of the Jacobian of the map $T$ at at $(0, \bar{y})$  is $1$.
The hyperbolic eigenvalue  is
$\frac{1}{a+\overline{y}}$,
with corresponding eigenvector
$(a - 1 + \overline{y} , \overline{y}\,(a+ \overline{y}) )$.
Thus the hypotheses of Theorem \ref{th: invariant curve} are satisfied.
Notice that the conditions of  Theorem \ref{thm: suff cond O+}
are satisfied and thus by Theorem \ref{thm: demottoni} all solutions
of (\ref{equation:compet2dclarkulb=1}) are eventually
componentwise monotone.
Indeed, the Jacobian matrix $J_{T}(x, y)$ satisfies
$\det J_{T}(x, y) = \frac{1}{a+\overline{y}} > 0$.
In addition, a direct verification shows that $T$ is injective.
A consequence of Theorem \ref{thm: demottoni} is that
there are no periodic points of minimal period-two.
Also, with an argument similar to the one used in \cite{BKK1},
one has that the equilibrium depends continuously on the initial
condition. That is, if $T^*(x,y) := \lim T^n(x,y)$, then
$T^*$ is continuous.
These considerations lead to the following result.
\begin{theorem}
\label{theorem:mkclarka>1b=1}
For system (\ref{equation:compet2dclarkulb=1}) with $a>1$,
\begin{itemize}
\item[\rm i.]
Every solution converges to an equilibrium $(0,\overline{y})$
for some $\overline{y}\geq 0$.
\item[\rm ii.]
At each equilibrium  $(0, \bar{y})$ with $\overline{y} > 0$,
the stable set $W^s((0,\overline{y}))$
is an unbounded increasing  curve $\mathcal{C}$
that starts at $(0, \bar{y})$.
 \item[\rm iii]
 The limiting equilibrium varies continuously with the initial condition.
\end{itemize}
\end{theorem}
Statement ii excludes the equilibrium $(0,0)$ since the hypotheses
of Theorem \ref{th: invariant curve} are not satisfied at $(0,0)$.
Theorem \ref{theorem:mkclarka>1b=1}
has been proved in \cite{CK} by using
 differential equations associated to the map $T$
and asymptotic estimates of infinite products.
\end{example}
\begin{example}
A system with a continuum of non-hyperbolic equilibria along
a line.
\rm
Consider the Leslie-Gower competition model with nonhyperbolic equilibrium points
\begin{equation}
\label{eq: Cushing}
\begin{array}{rcl}
x_{n+1} & = &  \displaystyle  \frac{b_1 \, x_n}{1+   x_n + c_{1}\, y_n} \\ \\
y_{n+1} & = &  \displaystyle  \frac{b_2 \, y_n}{1+ y_n + c_{2} \, x_n}
\end{array}
\quad , \quad n=0,1,\ldots
\end{equation}
where all parameters are positive and the initial conditions $x_0, y_0$ are non-negative. This system was considered in \cite{CLCH, KUM1} and its global dynamics has been settled with the exception of the nonhyperbolic case which will be considered here.
It is shown in \cite{CLCH} that when
$c_1\,(b_2-1) \neq b_1-1$
or
$c_2\,(b_1-1) \neq b_2-1$,
the map associated to (\ref{eq: Cushing}),
\begin{equation}
T(x,y) = \left(
\frac{b_1\, x}{1+x+c_1\, y} , \frac{b_2\, y}{1+y+c_2\, x}
\right)
\end{equation}
has between one and four fixed points,
and that they are of hyperbolic type.
The case when
\begin{equation}
\label{eq: cushing nonhyp}
c_1\,(b_2-1) = b_1-1
\quad \mbox{\rm and} \quad
c_2\,(b_1-1)= b_2-1
\end{equation}
was not considered in \cite{CLCH}.
When (\ref{eq: cushing nonhyp}) holds,
a direct calculation gives that
the equilibrium points of $T$ are $E_0(0,0)$ and
the family of points $\mathcal{E} :=\{ E_t \, : \, 0 \leq t \leq 1 \, \}$, where
$$
E_t := \left( (b_1 -1)\,(1-t)  ,  (b_2-1)\,t  \right), \quad 0 \leq t \leq 1.
$$
The eigenvalues of the Jacobian of $T$ at  $E_t$ are easily calculated to be
$$
1 \quad \mbox{and} \quad (1-t)\,\frac{1}{b_1} + t\,\frac{1}{b_2}\ ,
\quad 0 \leq t \leq 1\, ,
$$
and corresponding  eigenvectors are
$$
\left(- \frac{1-b_1}{1-b_2},1\right)
 \  \mbox{and} \
 \left(
 b_2\, (1-b_1)^2\,(1-t)\, , \, b_1\,(1-b_2)^2\, t \,
 \right)\ ,
\quad 0 \leq t \leq 1.
$$
It is shown in \cite{CaKuLaMe} that,
for system (\ref{eq: Cushing}),
the hypotheses
of  Theorem \ref{thm: demottoni}
are satisfied and that all solutions  fall inside an
invariant rectangular region.
Therefore  every solution of (\ref{eq: Cushing})
converges to an equilibrium point.
A direct calculation shows that the origin is a repeller.
We conclude
that every nonzero solution converges to a point
$(\overline{x},\overline{y}) \in \mathcal{E}$.
Also, with an argument similar to the one used in \cite{BKK1},
one has that the equilibrium depends continuously on the initial
condition. That is, if $T^*(x,y) := \lim T^n(x,y)$, then
$T^*$ is continuous.
These observations, together with
an application of Theorem \ref{th: invariant curve} lead to the following result.
\begin{theorem}
Assume inequalities (\ref{eq: cushing nonhyp}) hold. Then,
\begin{itemize}
\item[\rm  i]
Every nonzero solution to system (\ref{eq: Cushing}) converges to an equlibrium
$(\overline{x},\overline{y}) \in \mathcal{E}$.
\item[\rm ii]
For every $(\overline{x},\overline{y}) \in \mathcal{E}$ with $\overline{x} \neq 0$
and $\overline{y}\neq 0$, the stable set $W^s_{(\overline{x},\overline{y})}$ is
an unbounded increasing
curve $\mathcal{C}$ with endpoint $(0,0)$.
 \item[\rm iii]
 The limiting equilibrium varies continuously with the initial condition.
\end{itemize}
\end{theorem}
Statement ii excludes  equilibria of the form
$(0,\overline{y})$ and $(\overline{x},0)$  since the hypotheses
of Theorem \ref{th: invariant curve} are not satisfied  at these points.
\end{example}

\begin{example}
\label{ex: period two}
A difference equation with a continuum of period-two points along a branch of a hyperbola.
\rm
Consider the second order difference equation
\begin{equation}
\label{eq: period2trichamleh}
x_{n+1} = 1 + \frac{x_{n-1}}{x_n}, \quad n=0,1,2,\ldots\, ,
\end{equation}
where  the initial conditions $x_{-1}, x_0$ are positive. This equation was considered in \cite{AGGL, BKK1, KUL} and its global dynamics has been settled completely in \cite{AGGL}. Here we give more precise description of the dynamics of (\ref{eq: period2trichamleh}).
The map
\begin{equation}
T(x,y) = \left(
y , 1 + \frac{x}{y}
\right)
\end{equation}
 associated to (\ref{eq: period2trichamleh})
has a unique fixed point $(2, 2)$. The second iterate of $T$,
\begin{equation}
T^2(x,y) = \left(
1 + \frac{x}{y} , 1 + \frac{y^2}{x + y}
\right)\, ,
\end{equation}
is strongly competitive in the interior of first quadrant
and has an infinite number of fixed points $(\overline{x},\overline{y})$.
The collection of fixed points of $T^2$ (=period-two points of $T$) is the set
$$\mathcal{H} = \{\, (x,y) \in (0,\infty)^2\, :\, x + y = x\, y \, \}\, .
$$
The eigenvalues of $T^2$ at  $(\overline{x},\overline{y})\in \mathcal{H}$ are
$1$ and
$\frac{1}{\overline{x}\; \overline{y}} < 1$,
and the eigenvector of $T^2$ at  $(\overline{x},\overline{y})$ associated with $\lambda = \frac{1}{\overline{x}\; \overline{y}}$ is $(\overline{x}, 1)$.
The Jacobian matrix of $T^2$ is
$$
J_{T^2}(u,v) =    \left(
\begin{array}{cc}
\frac{1}{v}
         &
     -\frac{u}{v^2}
   \\ \\
   - \frac{v^2}{(u+v)^2}
      &
   \frac{v(2u + v)}{(u+v)^2}
         \end{array}
         \right)\ ,
$$
thus  $\det J_{T^2}(u,v) =\frac{1}{u + v} > 0$ for $(u,v) \in (0,\infty)^2$.
In addition, direct verification shows that $T^2$ is injective.
Thus all hypotheses of Theorem \ref{th: invariant curve}
are satisfied by $T^2$, so for every fixed point $(\overline{x},\overline{y})$ of $T^2$
(consequently, for every period-two point of $T$),
there exists an unbounded increasing invariant curve
$\mathcal{C}_{(\overline{x},\overline{y})}$ which is a subset of the basin of the attraction
of  $(\overline{x},\overline{y})$. Furthermore, it can be shown that all conditions of
deMottoni-Schiaffino theorem are satisfied and
so every solution of (\ref{eq: period2trichamleh})
converges to a period-two solution.
In addition, with an argument similar to the one used in \cite{BKK1},
applied to $T^2$, we may conclude
that, given any solution to Eq.(\ref{eq: period2trichamleh}),
the limiting period-two solution $(\overline{x},\overline{y})$
 depends continuously  on the initial condition $(x_0, y_0)$.
 That is, if $T^*(x,y) := \lim T^n(x,y)$, then
$T^*$ is continuous.
 Thus we have the following result.
\begin{theorem} The following statements are true for equation (\ref{eq: period2trichamleh}).
\begin{itemize}
\item[\rm i]
Every period-two solution converges to a period-two solution $(\overline{x},\overline{y})\in\mathcal{H}$.
\item[\rm ii]
For every period-two solution $(\overline{x},\overline{y}) \in \mathcal{H}$,
the stable set $W^s_{(\overline{x},\overline{y})}$ is
an unbounded increasing
curve $\mathcal{C}_{(\overline{x},\overline{y})}$.
 \item[\rm iii]
 The limiting period-two solution $(\overline{x},\overline{y})$
 depends continuously  on the initial condition $(x_0, y_0)$.
\end{itemize}
\end{theorem}

\end{example}

\begin{example}
A system with an isolated non-hyperbolic interior equilibrium
which is of oscillatory type.
\rm
The  system %
\begin{equation}
\label{eq: CAKULAME(15,21)}
\begin{array}{rcl}
x_{n+1} & = &  \displaystyle  \frac{\beta _{1}x_{n}}{B_{1}x_{n}+y_{n}} \\ \\
y_{n+1} & = &  \displaystyle  \frac{\alpha _{2}+\gamma _{2}y_{n}}{x_{n}}
\end{array}\ ,
\quad n=1,2,\ldots
\end{equation}
where all the parameters are positive and the initial conditions $x_0, y_0$ are non-negative and such that $x_0 + y_0 > 0$
was mentioned in \cite{CaKuLaMe} as a special case of system (\ref{system1}), and was investigated
in detail in \cite{GDKN1}.
When the condition
\begin{equation}
\label{eq: conditions eq: CAKULAME(15,21)}
\beta _{1}-B_{1}\gamma _{2}=2\sqrt{B_{1}\alpha _{2}}
\end{equation}
is satisfied, system (\ref{eq: CAKULAME(15,21)}) has a unique equilibrium point
$E=\left( \tfrac{B_{1}\gamma _{2}+\beta _{1}}{2B_{1}},\tfrac{\beta
_{1}-B_{1}\gamma _{2}}{2}\right)$  which is  nonhyperbolic. The eigenvalues of the linearized system at $E$ are
$$
\lambda_1 = 1, \quad
\lambda_2 =
-\frac{(\beta_1-B_1\,\gamma_2)^2}{2\,\beta_1\,B_1\,(\beta_1+B_1\,\gamma_2)}\, ,
$$
with corresponding eigenvectors
${\rm e}_1 = \left(\,-1\, ,\, B_1\,\right)$,
and
${\rm e}_2 =
\left(
2 \,\beta_1\,B_1\, (\beta_1-B_1\,\gamma_2\,)\, , \,
(\beta_1+B_1\,\gamma_2)^2
\right)$.
From (\ref{eq: conditions eq: CAKULAME(15,21)}) we have
$\lambda_2 \in (-1,0)$, and $E$ is of oscillatory type.
Thus the hypotheses of Theorems \ref{th: invariant curve} and \ref{th: basins}
are satisfied at the equilibrium point, and the conclusions of
Theorems \ref{th: invariant curve} and \ref{th: basins} follow.
Let $\mathcal{C}$, $\mathcal{W}_-$ and $\mathcal{W}_+$ be the sets
given in the conclusion of Theorems \ref{th: invariant curve} and \ref{th: basins}.
We have  the following result.
\begin{theorem}
The unique equilibrium $E$ of system (\ref{eq: CAKULAME(15,21)})
with conditions (\ref{eq: conditions eq: CAKULAME(15,21)}) is non-hyperbolic and semi-stable.
The basin of attraction of $E$ is $\mathcal{C} \cup \mathcal{W}_+$, and the orbit of every point in
$\mathcal{W}_-$ is asymptotic to $(0,\infty)$.
\end{theorem}
\begin{proof}
Let $S:=\{ (x,y) : 0 \leq x \leq \frac{\beta_1}{B_1}, \ 0 \leq y \ \}$.
Since $\frac{\beta_1\,x}{B_1\,x+y} \leq \frac{\beta_1}{B_1}$ for $x\geq 0$, $y \geq 0$, $x+y>0$,
the map $T$ of system (\ref{eq: CAKULAME(15,21)}) satisfies $T([0,\infty)^2\setminus (0,0) ) \subset S$.
Thus $T(\mathcal{C}\cup \mathcal{W}_+) \subset (\mathcal{C}\cup \mathcal{W}_+) \cap S$, which implies that
$T(\mathcal{C}\cup\mathcal{W}_+)$ is bounded.
Since every orbit is eventually coordinate-wise monotone \cite{CaKuLaMe},
it follows that every orbit with initial point in the invariant set
$\mathcal{C}\cup\mathcal{W}_+$ must converge to an equilibrium.
The only equilibrium in $\mathcal{C}\cup\mathcal{W}_+$ is $E$, so we have
$\mathcal{C}\cup\mathcal{W}_+$ is a subset of the basin of attraction of $E$.
If $(x,y)$ is in $\mathcal{W}_-$, by Theorem \ref{th: basins} the orbit of $(x,y)$
eventually enters $\mathcal{Q}_2(E)$.  Assume (without loss of generality)
that $(x,y) \in \mbox{\rm int }\mathcal{Q}_2(E)$.
A calculation gives
$$
T(E+ t \,{\rm e}_1)^{(1)} - (E+ t \,{\rm e}_1)^{(1)}  = 0 \quad \mbox{\rm for all } t
$$
and
$$
\frac{1}{2}\,\frac{d}{d\,t^2}\,\displaystyle \left. T(E+ t \,{\rm e}_1)\displaystyle \right|_{t=0} =
\left(0,\frac{1}{\beta_1+B_1\,\gamma_2}\right)\, .
$$
Since in expansion (\ref{eq: taylor}) we have 
$(c_2,d_2) =(0,\frac{1}{\beta_1+B_1\,\gamma_2})$
and
$T(E+ t \,{\rm e}_1)^{(1)}$ is affine in $t$,
by Theorem \ref{th: 2 and 4 all cases part 2} in any relative neighborhood
of $E$ there exists a subsolution $(w,z) \in Q_2(E)$, i.e.,
$T(w,z) \preceq_{se} (w,z)$.
Choose one such $(w,z)$ so that
$(x,y) \preceq_{se} (w,z)$.
Since $T$ is competitive,
  $T^{n+1}(w,z) \preceq T^n(w,z)$ for $n=0,1,2,\ldots$.
The monotonically decreasing sequence $\{T^n(w,z)\}$
in $\mathcal{Q}_2(E)$   is unbounded below,
since if it weren't it would converge to the unique fixed point
in $\mathcal{Q}_2(E)$, namely $E$, which is not possible.
Let $(w_n,z_n) := T^n(w,z)$, $n=0,1,\ldots$.
Then $(w_n,z_n) \in S$ for $n=1,2,\ldots$,
hence $\{w_n\}$ is bounded.
It follows that $\{z_n\}$
 is monotone and unbounded.
From (\ref{eq: CAKULAME(15,21)}) it follows that $w_n \rightarrow 0$.
Since $T^n(x,y) \preceq_{se} (w_n,z_n)$, it follows that
$T^n(x,y) \rightarrow (0,\infty)$.
\end{proof}
\end{example}

%
\begin{example}
A system with a semi-stable non-hyperbolic interior equilibrium.
\rm
We consider an example where explicit computation of
eigenvalues of the Jacobian of the map at equilibrium points is not practical.
Nevertheless, it is possible to establish a result on basins of attraction of
non-hyperbolic equilibria.
The strongly competitive system of difference equations
\begin{equation}
\label{eq: cushing+h}
\begin{array}{rcl}
x_{n+1} & = & \displaystyle \frac{b_1\, x_n}{1+x_n+c_1\, y_{n}} +h_1\\ \\
y_{n+1} & = & \displaystyle \frac{b_2\, y_n}{1+y_n+c_2\, x_{n}} +h_2
\end{array}
\, ,
\quad  n=0,1,\ldots, \quad (x_0,y_0) \in [0,\infty)\times [0,\infty)\, .
\end{equation}
with positive parameters was studied in \cite{BM},
where it was shown that the system has between one and three
equilibria depending on the choice of parameters,
and that the number of equilibria determines global behavior as follows:
if there is only one equilibrium, then it is globally asymptotically stable.
If there are two equilibria, then one is a locally asymptotically stable
(L.A.S.) point
 and the other one is nonhyperbolic.
If there are three equilibria,
then they are linearly ordered in the south-east ordering of the plane,
and consist of a  L.A.S.  point,
a saddle point, and another L.A.S. point.
We have the following result for the non-hyperbolic case.
\begin{theorem}
\label{th: cushing+h}
In the case where system (\ref{eq: cushing+h})
 has exactly two distinct equilibria in $[0,\infty)^2$ given by
a locally asymptotically stable point $(\overline{x}_1,\overline{y}_1)$
and a non-hyperbolic fixed point $(\overline{x}_2,\overline{y}_2)$,
there exists a curve $\mathcal{C}$ through $(\overline{x}_2,\overline{y}_2)$
as  in Theorem \ref{th: invariant curve} and sets $\mathcal{W}_-$ and $\mathcal{W}_+$
as  in Theorem \ref{th: basins}
such that
one of  $\mathcal{W}_-$ or $\mathcal{W}_+$
is  the basin of attraction of $(\overline{x}_1,\overline{y}_1)$,
and the complement of such set is the basin of attraction of
$(\overline{x}_2,\overline{y}_2)$
\end{theorem}
\begin{proof}
it is shown in \cite{BM} that
under the hypothesis of the theorem, the intersection of
$(h_1,\infty)\times (h_2,\infty)$
with each of the critical curves
\begin{equation*}
\label{eq: critical curves system}
\begin{array}{rcl}
C_{1}
&  =& \{(x,\,y) \in {[0,\infty)}^{\,2} : x^2 + c_1 \, x \, y +(1-b_1-h_1)\, x
- c_1\,h_1\,y-h_1 =  0 \}  \\
\\
C_{2}
&  =& \{ (x,\,y) \in {[0,\infty)}^{2} : y^2 + c_2 \, x \, y +(1-b_2-h_2)\, y-
c_2\,h_2\,x-h_2 =  0\}
\end{array}
\end{equation*}
are the graphs of smooth decreasing functions $y_1(x)$ and $y_2(x)$
of $x$ with
 common points given by the two equilibrium points
of system (\ref{eq: cushing+h}).
Moreover, $C_1$ and $C_2$  intersect
tangentially
at the non-hyperbolic fixed point $(\overline{x}_1,\overline{y}_1)$,
and  transversally at the local attractor
$(\overline{x}_2,\overline{y}_2)$.
By Lemma 4.2  of \cite{BM} and Theorem 2.1 of \cite{BM},
the eigenvalues $\lambda$, $\mu$
of the Jacobian of the map at
$(\overline{x}_1,\overline{y}_1)$ satisfy
$0 < |\lambda|<1$ and $\mu=1$.
Note that there are no periodic points of minimal period-two.
Thus the hypotheses of Theorem \ref{th: invariant curve}  and Theorem \ref{th:
basins} are satisfied, hence there exist sets $\mathcal{C}$,
$\mathcal{W}_-$ and $\mathcal{W}_+$ with the properties specified in the
conclusions
of such theorems.  In particular,
\begin{equation}
\label{eq: disjoint union}
[0,\infty)^2 = \mathcal{C} \cup \mathcal{W}_- \cup \mathcal{W}_+\, ,
\end{equation}
where any two of the sets $\mathcal{C} $, $\mathcal{W}_-$, and
$\mathcal{W}_+$
have no common points.
Since the set of equilibrium points is linearly ordered by $\preceq_{se}$,
it follows that either $(\overline{x}_2,\overline{y}_2) \in \mathcal{W}_-$ or
$(\overline{x}_2,\overline{y}_2) \in \mathcal{W}_+$.
In the rest of the proof we shall
assume that $(\overline{x}_2,\overline{y}_2) \in \mathcal{W}_-$;
the proof in the case $(\overline{x}_2,\overline{y}_2) \in \mathcal{W}_+$
is similar and will be omitted. Then
the only fixed point in the closed invariant set
$\mathcal{C}\cup \mathcal{W}_+$ is $(\overline{x}_1,\overline{y}_1)$,
and since every orbit converges to a fixed point  \cite{BM},
we have that $\mathcal{C}\cup \mathcal{W}_+$ is a subset of the basin of
attraction of $(\overline{x}_1,\overline{y}_1)$.
To complete the proof of the theorem
it is enough to show that $\mathcal{W}_-$ is a subset of the basin of
$(\overline{x}_2,\overline{y}_2)$.
There exists an open interval $I$ centered at $x_1$
such that
$y_1(x) < y_2(x)$ for $x \in I \setminus \{ x_1 \}$.
Set $y_3(x) = \frac{1}{2}(y_1(x) + y_2(x))$, and let
$C_3$ be the graph of $y_3(x)$.
For  small enough $\delta$, the map $T$ satisfies
$T(x,y) \preceq_{se} (x,y)$ for every $(x,y) \in C_3 \cap \mathcal{B}(\,
(\overline{x}_1,\overline{y}_1),\delta)$, where
$\mathcal{B}(\,(\overline{x}_1,\overline{y}_1),\delta)$ is the open disk in $\mathbb{R}^2$
with center $(\overline{x}_1,\overline{y}_1)$  and radius $\delta$.
By Theorem \ref{th: basins}
the orbit of every $(x,y) \in \mathcal{W}_-$ enters the invariant set
$\mathcal{Q}_2(x_2,y_2) \cap (0,\infty)^2$.
Choose $(s,t) \in C_3 \setminus \{(x_1,y_1)\}$
such that $(x,y) \preceq_{se} (s,t) \preceq_{se} (x_1,y_1)$.
Since $T^n(s,t) \preceq_{se} T^n(s,t) \preceq_{se} (s,t)$ for $n=0,1,2,\ldots$,
we have that $T^n(x,y)$ does not converge to $(x_1,y_1)$.
Hence $T^n(x,y)$ converges to $(x_2,y_2)$.
\end{proof}

\end{example}

\section{Proof of the main results}
\label{sec: Proof}
{\bf Proof of Theorem \ref{th: invariant curve}}.
We assume  first that $\overline{\rm x} \in \partial \mathcal{R}$.
Suppose (without loss of generality) that
$Q_1(\overline{\rm x})\cap \mbox{\rm int}(\mathcal{R}) \neq \emptyset$.
See Figure  \ref{fig: boundary case figures}.
%
%
\begin{figure}[!h]
\begin{center}
\setlength{\unitlength}{0.4 in}
\begin{picture}(16,4)
\qbezier(0,1)(1, 1.5)(2,3)
\put(0.5,1.87){$\mathcal{C}$}
\put(0.25,0.45){$\overline{\rm x}$}
\put(2,1.85){$\mathcal{Q}_1(\overline{\rm x})$}
\put(2,0.35){$\mathcal{Q}_4(\overline{\rm x})$}
\put(0,1){\circle*{0.2}}
\multiput(0,1)(0.6,0.0){7}{\line(1,0){0.3}}
\put(0,0){\line(1,0){4}}
\put(4,0){\line(0,1){3}}
\put(0,0){\line(0,1){3}}
\put(0,3){\line(1,0){4}}
\qbezier(6,0)(8, 0.5)(9,3)
\put(7.8,1.80){$\mathcal{C}$}
\put(6.25,0.40){$\overline{\rm x}$}
\put(8.5,0.85){$\mathcal{Q}_1(\overline{\rm x})$}
\put(6,0){\circle*{0.2}}
\put(6,0){\line(1,0){4}}
\put(10,0){\line(0,1){3}}
\put(6,0){\line(0,1){3}}
\put(6,3){\line(1,0){4}}
\qbezier(13.5,0)(15.5, 0.5)(16,3)
\put(15.2,0.65){$\mathcal{C}$}
\put(13.,0.25){$\overline{\rm x}$}
\put(12.25,1.85){$\mathcal{Q}_2(\overline{\rm x})$}
\put(14,1.85){$\mathcal{Q}_1(\overline{\rm x})$}
\put(13.5,0){\circle*{0.2}}
\multiput(13.5,0.2)(0.0,0.6){5}{\line(0,1){0.3}}
\put(12,0){\line(1,0){4}}
\put(16,0){\line(0,1){3}}
\put(12,0){\line(0,1){3}}
\put(12,3){\line(1,0){4}}
\end{picture}
\end{center}
\caption{The three cases where the fixed point
$\overline{\rm x} \in \partial \mathcal{R}$ and
$Q_1(\overline{\rm x}) \cap \mathcal{R} \not =\emptyset$.}
\label{fig: boundary case figures}
\end{figure}
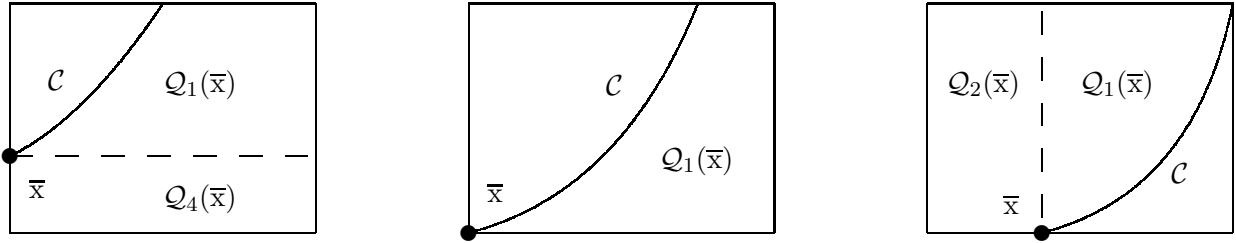
%
%
%

By Lemma 5.1 in page 234 of \cite{Hartman},
there exists a small neighborhood $\mathcal{V}$ of $\overline{\rm x}$
and a locally invariant $C^1$ manifold
$\mathcal{\hat{C}}\subset \mathcal{V}$
that is tangential to $E^\lambda$ at $\overline{\rm x}$
and such that
$T^n({\rm x}) \rightarrow \overline{\rm x}$ for all ${\rm x} \in \mathcal{\hat{C}}$.
Since $T$ is competitive, a unit eigenvector ${\rm v}^\lambda$ associated with
$\lambda$ may be chosen so that ${\rm v}^\lambda$ has non-negative entries.
By hypothesis (b.), the vector ${\rm v}^\lambda$ has positive entries.
If necessary, the diameter of $\mathcal{V}$ may be taken to be small enough
to guarantee that
no two points on $\mathcal{\hat{C}}$
are comparable in the ordering $\preceq_{se}$.
This can be done
since the tangential vector ${\rm v}^\lambda$ has positive entries.
%
 Let $\mathcal{C}$ be the connected component of
 $\{{\rm x} \in \mathcal{R}\cap
 \mathcal{Q}_1(\overline{\rm x}) : (\exists n) \ T^n({\rm x}) \in \mathcal{\hat{C}}\}$
 that contains $\mathcal{\hat{C}}$.
 The set  $\mathcal{C}$  consists of non-comparable
 points in the order $\preceq_{se}$.
 Indeed,  if $\rm v$ and $\rm w$ are two distinct   points in $\mathcal{C}_+$ such that
 ${\rm v} \preceq_{se} {\rm w}$, then $T^n({\rm v}) \preceq_{se} T^n({\rm w})$
 and $T^n({\rm v}) \neq T^n({\rm w})$ for $n \geq 0$
 since the map $T$ is strongly competitive.
 But for $n$ large enough, both $T^n({\rm v})$ and $ T^n({\rm w})$ belong to $\mathcal{\hat{C}}$,
 which consists of non-comparable points. Hence $\mathcal{C}$ consists of non-comparable points.
 The projection of $\mathcal{C}$ onto the first coordinate
 is a connected set, thus it is an interval $J \subset \mathbb{R}$.
 Since points on $\mathcal{C}$ are non-comparable, $\mathcal{C}$ is the graph
 of a strictly increasing function $f(t)$ of $t \in J$.
 If there is a jump discontinuity at $t_0 \in \mathcal{C}$, let $y_-$ and $y_+$ respectively
 be the left and right (distinct) limits of $f(t)$ as $x$ approaches $t_0$, respectively.
 The points $(t_0,y_-)$ and $(t_0,y_+)$ are comparable in the order $\preceq_{se}$,
 and since $T$ is strongly competitive in $\Delta$, $T(t_0,y_-)\preceq_{se}T(t_0,y_+)$
 and  $T(t_0,y_+)-T(t_0,y_-) \in {\rm int}\,\mathcal{Q}_4(\overline{0})$.
 Since both $T(t_0,y_+)$ and $T(t_0,y_+)$ are accumulation points of $\mathcal{C}$,
 we obtain that $\mathcal{C}$  must have comparable points, a contradiction.
 Thus $f(t)$ is a continuous function.

If $\mathcal{C}$ is not bounded, then both of its endpoints are in $\partial \mathcal{R}$,
as the conclusion of the theorem asserts.
If $\mathcal{C}$ is bounded, it has two endpoints,
$\overline{\rm x}$ and ${\rm x}_0$ (say).
To show that either  ${\rm x}_0 \in \partial \mathcal{R}$, or
 ${\rm x}_0$ is a fixed point of $T$,
assume this is not the case.  That is, assume
${\rm x}_0 \in {\rm int}\,\mathcal{R}$ and $T({\rm x}_0) \neq {\rm x}_0$.
We shall show that in this case, the curve $\mathcal{C}$ can be extended,
contradicting the definition of $\mathcal{C}$ as a connected component.

We first show that $T({\rm x})\neq\overline{\rm x}$ for ${\rm x} \in \Delta$.
To see this, consider points ${\rm y}$ and ${\rm z}$ in $\Delta$
so that the points ${\rm y}$, ${\rm x}_0$ and ${\rm z}$ lie on a vertical line $x=c$
with ${\rm y} \preceq_{se} {\rm x}_0 \preceq_{se} {\rm z}$, and such that
both ${\rm y}$ and $\rm z$ distinct from ${\rm x}_0$.
Then $T({\rm y}) \preceq_{se} T({\rm x}_0) \preceq_{se} T({\rm z})$.
Furthermore, since $T$ is strongly competitive,
if $T({\rm x}_0)=\overline{\rm x}$, we have
${\rm x}_0-T({\rm x})\in {\rm int}\,\mathcal{Q}_4(\overline{\rm x})$ and
$T({\rm z})- {\rm x} \in {\rm int}\,\mathcal{Q}_4(\overline{\rm x})$.
Since $\overline{\rm x} \in \partial \mathcal{R}$, we conclude that one of the points
$T({\rm y})$ and $T({\rm z})$ does not belong to $\mathcal{R}$, a contradiction.

Since $T({\rm x}_0)\neq \overline{\rm x}$, $T({\rm x}_0)\neq {\rm x}_0$ and
$\mathcal{C}$ is a (forward) invariant connected set,
we must have $T({\rm x}_0) \in \mathcal{C}\cap \Delta$.
  Let $\mathcal{R}_1$ be the rectangular region determined by
  the endpoints of $\mathcal{C}$,
  and
  let $\varepsilon>0$ be such that $\mathcal{B}(T({\rm x}_0),\varepsilon) \subset
  \mathcal{R}_1$. Note that $\mathcal{C}$ is a separatrix for $\mathcal{R}_1$.
  For ${\rm y} \in \mathbb{R}^2$ and $\eta>0$,
  denote with $\mathcal{B}({\rm y},\eta)$ the open disk in $\mathbb{R}^2$
   with center $\rm y$ and radius $\eta$.
  By continuity of $T$, there exists $\delta>0$ such that
  $T(\mathcal{B}({\rm x}_0,\delta)) \subset \mathcal{B}(T({\rm x}_0),\varepsilon)$.
  Consider the line segment $L_0$ with endpoints $({\rm x}_0^{(1)}, {\rm x}_0^{(2)}\pm\delta/2)$.
  Since $L_0$ is linearly ordered by $\preceq_{se}$, so is $T(L_0)$, and
  the points $T({\rm x}_0^{(1)}, {\rm x}_0^{(1)}\pm\delta/2)$ are on different components
  of $\mathcal{R}_1\setminus \mathcal{C}$.
  Find $\varepsilon_2>0$
  such that $\mathcal{B}(T({\rm x}_0^{(1)}, {\rm x}_0^{(2)}\pm\delta/2),\varepsilon_2)$
  is a subset of a component of
  $\mathcal{R}_1\setminus \mathcal{C}$ and of
  $\mathcal{B}(T({\rm x}_0),\varepsilon)$, and choose $\eta>0$ such that
  $$
  T[\mathcal{B}(({\rm x}_0^{(1)},{\rm x}_0^{(2)}\pm\delta/2) , \eta )] \subset
  \mathcal{B}(T[({\rm x}_0^{(1)},{\rm x}_0^{(2)}\pm\delta/2)],\varepsilon_2)
  $$
  Now for each $t \in (0,\eta)$ consider the line segment  $L_t$ with  endpoints
  ${\rm p}_\pm(t) := ({\rm x}_0^{(1)}+t,{\rm x}_0^{(2)}\pm \delta/2)$.
   Note that $L_t \subset \mathcal{B}({\rm x}_1,\delta)$, and that
   $T({\rm p}_+(t))$, $T({\rm p}_-(t))$  belong to
  different components of $\mathcal{R}_1\setminus \mathcal{C}$.
  For each $t \in (0,\eta)$, the line segment $L_t$  is linearly ordered by $\preceq_{se}$, hence so is $T(L_t)$.
  Thus for each $t\in (0,\eta)$  there exists   ${\rm x}_t \in L_t$ such that $T({\rm x}_t) \in \mathcal{C}$.
  That is, the function $f(x)$ may be extended to a function $\hat{f}(x)$ defined on an interval $\hat{J}$
  that includes $J$ as a proper subset.  The reasoning used to show continuity and monotonicity of $f(x)$ gives
  continuity and monotonicity of $\hat{f}(x)$.  This contradicts the choice of $\mathcal{C}$ as a component,
  and we conclude that either $T({\rm x}_0)={\rm x}_0$ or ${\rm x}_0\in \partial \mathcal{R}$.   See Figure \ref{fig: proof of theorem 1}.

We have proved the theorem for $\overline{\rm x} \in \partial \mathcal{R}$.
If now $\overline{\rm x} \in {\rm int}\, \mathcal{R}$,
the argument in the case  $\overline{\rm x} \in \partial \mathcal{R}$
may be used to demonstrate the existence of a curve $\mathcal{C}$
 with endpoints
${\rm w} \in {\rm int}\,\mathcal{Q}_3(\overline{\rm x})$ and
${\rm z} \in {\rm int}\,\mathcal{Q}_1(\overline{\rm x})$.
By the argument used before, if  both of ${\rm w},{\rm z}$ are
in the interior of $\mathcal{R}$, then  the set $\{ {\rm w} , \rm z\}$
is invariant and therefore consists of fixed points or of minimal period-two points.
If only one of ${\rm w}$, ${\rm z}$ is in ${\rm int}\, \mathcal{R}$, then such point
must be a fixed point of $T$.
\hfill $\Box$
\bigskip

\begin{figure}[!h]
\begin{center}
\setlength{\unitlength}{1 in}
\begin{picture}(4,3)
\put(0,-0.75){\includegraphics[width=4in]{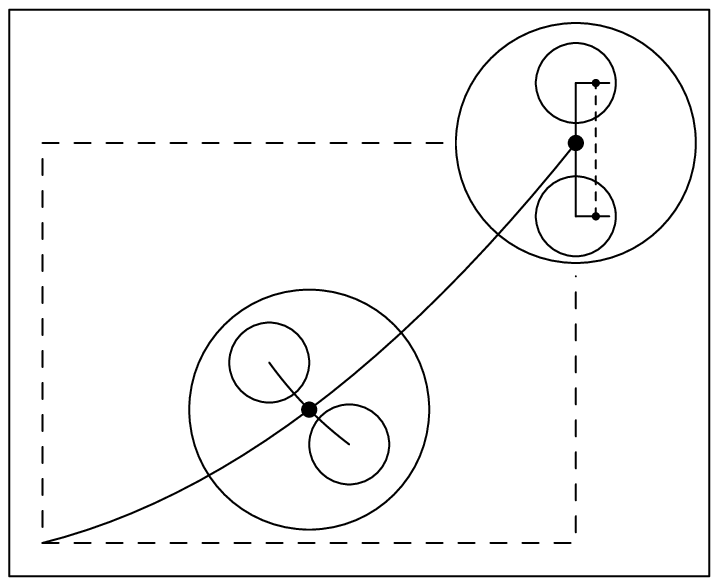}}
\put(0.6,2.7){$\mathcal{R}$}
\put(0.6,2.0){$\mathcal{R}_1$}
\put(0.6,0.5){$\mathcal{C}$}
\put(2.78,2.22){${\rm x}_0$}
\put(3.4,2.2){$L_t$}
\put(3.35,2.2){\vector(-1,0){0.24}}
\put(1.75,1.3){\small $T(L_0)$}
\put(1.51,0.6){\vector(1,2){0.13}}
\put(1.25,0.5){\small $T({\rm x}_0)$}
\put(1.74,1.3){\vector(-1,-1){0.2}}
\put(2.0,1.9){\small $\mathcal{B}({\rm x}_0,\delta)$}
\put(2.3,0.78){\small $\mathcal{B}(T({\rm x}_0),\epsilon)$}
\put(0.55,1.6){\vector(2,-1){0.74}}
\put(2.35,2.6){\vector(1,0){0.45}}
\put(1,2.55){\small $\mathcal{B}(({\rm x}_0^{(1)}, {\rm x}_0^{(2)}+\delta/2),\eta)$}
\put(0.5,1.65){\small $\mathcal{B}(T({\rm x}_0^{(1)}, {\rm x}_0^{(2)}+\delta/2),\varepsilon_2)$}
 \end{picture}
\end{center}
\caption{Sets that appear in the proof of Theorem \ref{th: invariant curve}.}
\label{fig: proof of theorem 1}
\end{figure}
\bigskip

\noindent {\bf Proof of Theorem \ref{th: suff cond for boundary endpoints}}.
i. The conclusion follows from Theorem \ref{th: invariant curve}.
ii.
We claim that $\mathcal{C}\cap \mathcal{ Q}_{1}(\overline{\rm x})$
and $\mathcal{C}\cap \mathcal{ Q}_{3}(\overline{\rm x})$
are invariant.  Note first that both of these sets are connected, hence so are
their images under $T$.
Since $T({\rm x}) = \overline{x}$ for ${\rm x} \in \mathcal{C}$
is not possible by hypothesis, it follows that for $\ell=1,2$,
$\overline{\rm x}$ is necessarily an endpoint of
$T(\mathcal{C}\cap \mathcal{ Q}_{\ell}(\overline{\rm x}))$.
Hence  either
$T(\mathcal{C} \cap \mathcal{ Q}_{\ell}(\overline{\rm x})) \subset \mathcal{ Q}_{1}(\overline{\rm x})$, or,
$T(\mathcal{C} \cap \mathcal{ Q}_{\ell}(\overline{\rm x})) \subset \mathcal{ Q}_{3}(\overline{\rm x})$, $\ell=1,2$.
We now show
$T(\mathcal{C} \cap \mathcal{ Q}_{1}(\overline{\rm x})) \subset \mathcal{ Q}_{1}(\overline{\rm x})$.
Since $T$ is competitive, the largest eigenvalue $\mu$ of
$J_T(\overline{\rm x})$ is positive.
This fact and the hypothesis
$\det J_T(\overline{\rm x}) > 0$ implies that both eigenvalues
$\lambda$, $\mu$, of $J_T(\overline{\rm x})$ are positive.
Let ${\rm v}^\lambda$ be an eigenvector associated with $\lambda$.
Since the product of the entries of ${\rm v}^\lambda$ is positive,
we may assume withtout loss of generality the entries are positive, and in this case
for all $r>0$  small enough,
${\rm x} = \overline{\rm x} + r {\rm v}^\lambda
\in \Delta \cap \mathcal{Q}_1(\overline{x})$.
Now ${\rm x} \in \mathcal{C}$ satisfies
$T({\rm x}) = \overline{\rm x} + \lambda\, r\, {\rm v}^\lambda + o(r)$,
hence for $r>0$ small, $T({\rm x}) \in  \mathcal{C} \cap \mathcal{Q}_1(\overline{\rm x})$.
This proves $ \mathcal{C} \cap \mathcal{Q}_1(\overline{\rm x})\neq \emptyset$, and therefore
$ \mathcal{C} \cap \mathcal{Q}_1(\overline{\rm x}) \subset  \mathcal{Q}_1(\overline{\rm x})$.
That $ \mathcal{C} \cap \mathcal{Q}_3(\overline{\rm x}) \subset  \mathcal{Q}_3(\overline{\rm x})$
 is proved in similar manner.
The statement that $\mathcal{C}$ has no endpoints in the interior of $\mathcal{R}$
follows from the fact
that $\mathcal{C}\cap \mathcal{ Q}_{1}(\overline{\rm x})$ and
$\mathcal{C}\cap \mathcal{ Q}_{3}(\overline{\rm x})$ are invariant.  Indeed,
each one of these sets has a set of endpoints which is invariant, and contains $\overline{\rm x}$.
Since $T({\rm x}) = \overline{\rm x}$ only for ${\rm x}=\overline{{\rm x}}$
by hypothesis, the endpoint that is not equal to $\overline{\rm x}$ must be
a fixed point, a contradiction.
The proof of iii. is similar to the proof of ii. and we skip it.
\hfill $\Box$
\bigskip


\noindent {\bf Proof of Theorem \ref{th: smooth manifold}}.
By Lemma 5.1 and Exercise 5.1 (a) (ii) in pages 234 and  238 of \cite{Hartman},
there exists a neighborhood $\mathcal{D}$ of $\overline{\rm x}$ 
such that $\mathcal{D}\cap\mathcal{C}$ is a class $C^k$ manifold.
For arbitrary ${\rm x}\in \mathcal{C}$, let 
$n\in\mathbb{N}$ be such that $T^n({\rm x}) \in \mathcal{D}\cap \mathcal{C}$. 
By the hypotheses on $T$, the map $T^n$ is of class $C^k$, with $C^k$ inverse
defined on a neighborhood $\mathcal{E}$ of  $T^n({\rm x})$.
 Thus $(T^n)^{-1}(\mathcal{E}\cap \mathcal{C})$ is a $C^k$ manifold.
\hfill $\Box$

\bigskip


\noindent {\bf Proof of Theorem \ref{th: basins}}.
For convenience, in this proof we assume $\overline{\rm x} = \overline{0}$.
It is clear that the sets
$\mathcal{W}_- $
and
$\mathcal{W}_+ $
are connected, disjoint,
and satisfy $\mathcal{C} \cup \mathcal{W}_- \cup \mathcal{W}_+ \subset \mathcal{R}$.
To prove the reverse inclusion, let $\rm w$ and $\rm z$ be the two endpoints of $\mathcal{C}$,
where ${\rm w } \preceq_{ne} {\rm z}$.  If ${\rm x} \in \mathcal{R}\setminus ( \mathcal{C} \cup \mathcal{W}_- \cup \mathcal{W}_+)$,
then either ${\rm w} \in \mathbb{R}^2$ and ${\rm x} \preceq_{ne} {\rm w}$, or
${\rm z} \in \mathbb{R}^2$ and ${\rm z} \preceq_{ne} {\rm x}$.
Suppose ${\rm z} \in \mathbb{R}^2$ and ${\rm z} \preceq_{ne} {\rm x}$.
Since ${\rm z} \in \partial \mathcal{R}$, then ${\rm x} \in \partial \mathcal{R}$
and necessarily $\rm x$ is comparable to $\rm z$.
Since $z\in \mathcal{C}$ contradicts the assumption on $\rm x$,
necessarily ${\rm z} \not \in \mathcal{C}$.  But then $z \not \in \mathcal{R}$,
which implies that ${\rm x} \not \in \mathcal{R}$.
This contradiction proves the inclusion.

 We now prove (i).
 If ${\rm x} \in \mathcal{W}_-$, let ${\rm y} \in  \mathcal{C}$ be such that ${\rm x} \preceq_{se} {\rm y}$.
 Then $T^n({\rm x}) \preceq_{se} T^n({\rm y})$ for $n >0$.
 Hence  $T^n({\rm x})^{(1)} \leq T^n({\rm y})^{(1)} $,
 and
 $T^n({\rm x})^{(2)} \geq T^n({\rm y})^{(2)} $,
where for ${\rm v} \in \mathbb{R}^2$ we write ${\rm v}=({\rm v}^{(1)},{\rm v}^{(2)})$.
 Since $T^n({\rm y}) \rightarrow \overline{0}$ as $n\rightarrow \infty$,
 it follows that
 $\lim \sup T^n({\rm x})^{(1)} \leq \lim \sup T^n({\rm y})^{(1)} = 0$
 and
 $\lim \inf T^n({\rm x})^{(2)} \geq \lim \inf T^n({\rm y})^{(2)} = 0$,
 that is,
 $\mbox{\rm dist}(T^n({\rm x}),\mathcal{Q}_2(\overline{0}))\rightarrow 0$.
 This proves (i).  The proof of (ii) is similar.

To prove (iii), assume first $\lambda > 0$ and
$\mbox{\rm int}\, \mathcal{Q}_1(\overline{0}) \cap \mathcal{R} \not = \emptyset$.
We proceed by contradiction and assume
there exists $(x_0,y_0) \in \mathcal{W}_-$
such that
\begin{equation}
\label{eq: assumption for contrad}
T^n(x_0,y_0)   \in {\rm int}\,\mathcal{Q}_1(\overline{0})
\quad \mbox{for } n \in \mathbb{N}\, .
\end{equation}

\begin{claim}
$T^n(x_0,y_0) \rightarrow (0,0)$.
\end{claim}
To prove the claim, let $\eta>0$ be such that $T$ is strongly monotonic
on $\mathcal{B}(\overline{0},\delta)$, and
let $\epsilon$ be an arbitrary positive number in $(0,\eta)$.
Since $(0,\epsilon) \preceq_{se} (0,0)$,
strong monotonicity of $T$ implies
$T(0,\epsilon) \in \mbox{\rm int}\,\mathcal{Q}_2((0,0))$.
By continuity of $T$, there exists $\delta(\epsilon) > 0$ such that
$(s,t) \in\, {\rm clos}\,\mathcal{B}((0,\epsilon_0),\delta(\epsilon))$ implies
$T(s,t) \in \mbox{\rm int}\,\mathcal{Q}_2((0,0))$.
In particular, $T(\delta(\epsilon),\epsilon) \in
 \mbox{\rm int}\,\mathcal{Q}_2((0,0))$.
If $(s,t) \in S(\delta(\epsilon),\epsilon) := \{ (w,z): 0 \leq w < \delta(\epsilon),
\epsilon < z \}$, then $(s,t) \preceq_{se} (\delta(\epsilon),\epsilon)$ and
$T(s,t) \preceq_{se} T(\delta(\epsilon),\epsilon)$ since $T$ is competitive.
Thus $T(S(\delta(\epsilon),\epsilon) ) \subset \mbox{\rm int}\,\mathcal{Q}_2((0,0))$.
Since $(x_n,y_n) :=T^n(x_0,y_0)$ satisfies $y_n \rightarrow 0$ by part (i) of the Theorem,
and since $(x_n,y_n) \not \in S(\delta(\epsilon),\epsilon)$ for all $\epsilon \in (0,\eta)$,
we conclude $y_n \rightarrow 0$, thus completing the proof of the claim.\hfill $\Box$
\medskip

By Lemma 5.1 and Exercise 5.1 (a) (ii) in pages 234 and  238 of \cite{Hartman},
there exists a $C^2$ change of coordinates $\Theta$
such that  the map $\hat{T}:= \Theta T \Theta^{-1}$
is defined in a neighborhood $\mathcal{B}(\overline{0},\hat{\delta})$ of $\overline{0}$
where it is of class $C^2$, $\hat{T}$ has eigenvectors $\mu$ and $\lambda$
with associated eigenvectors $(0,1)$ and $(1,0)$, and such that
the invariant curve $\mathcal{C}$ is mapped to the $x$-axis.
The points $(\hat{x}_n,\hat{y}_n):=\Theta(x_n,y_n) = \hat{T}^n(\hat{x}_0,\hat{y}_0)$,
satisfy
\begin{equation}
\label{eq: hat xn}
(\hat{x}_n,\hat{y}_n) \in \mathcal{Q}_1(\overline{0}),\ n=0,1,2,\ldots ,
\quad \mbox{\rm and} \quad
(\hat{x}_n,\hat{y}_n) \rightarrow \overline{0}\, .
\end{equation}

For $m>0$, let $\Omega(m)$ be the open wedge in the first quadrant
limited by the x-semiaxis and the line $y=m\,x$, that is,
$$ \Omega(m) = \{ (x,y) \in (0,\infty)^2 : 0 < y < m\, x \}\, .
$$
If ${\rm v} \in Q_1(\overline{0})$, by $\theta({\rm v})$
we denote the measure of the polar angle of ${\rm v}$
with respect to the positive horizontal semiaxis.
Thus $\tan \theta({\rm v}) = {\rm v}^{(2)}/{\rm v}^{(1)}$ whenever ${\rm v}^{(1)}\neq 0$.
\begin{claim}
For every $m\in (0,\infty)$ and every $q \in (1,\frac{\mu}{\lambda})$
there exists $\delta>0$ such that
\begin{equation}
\label{ineq: lemma q}
\hat{T}(x,y) \in \mbox{\rm int}\, Q_1(\overline{0})
\quad \mbox{\rm and } \quad
\tan \theta(\hat{T}(x,y)) > q\, \tan\,\theta(x,y)
\quad \mbox{\rm for} \quad
(x,y) \in \Omega(m) \cap \mathcal{B}(\overline{0},\delta)\, .
\end{equation}
\end{claim}
%
\begin{proof}
Let $A$ be the Jacobian operator of $\hat{T}$ at $\overline{0}$.
Thus the matrix representation of $A$ in the standard basis of $\mathbb{R}^2$ is
a diagonal matrix with $\lambda$ and $\mu$ on the diagonal.
Since $\hat{T}$ is $C^2$ on a neighborhood of the origin,
Lemma 10.11 from \cite{ASY} guarantees
 that for every $\epsilon>0$ there exists $\delta>0$ such that
$$
\|\, \hat{T}(x,y)-\hat{T}(s,t) - A\,(x-s,y-t)\,  \| < \epsilon \|\, (x-s,y-t)\,  \| \quad \mbox{\rm for } (x,y),(s,t) \in \mathcal{B}(0,\delta)\, .
$$
Then for every $\epsilon>0$ there exists $\delta_1>0$ such that
\begin{equation}
\label{ineq: c1}
\|\hat{T}(x,y) - \hat{T}(x,0)-\mu\,y\,(0,1) \| < \varepsilon \, |\,y\,|
\quad \mbox{\rm for} \quad
(x,y)  \in \mathcal{B}(\overline{0},\delta_1)\, .
\end{equation}
By continuity of the function $g(s,t) := \frac{\mu+s}{\lambda+t}$ at $(0,0)$
and since $g(0,0) = \frac{\mu}{\lambda} > q > 1$,
there exists $\eta  \in (0,\lambda)$ such that
\begin{equation}
\label{eq: eta}
r(s,t) > q \quad \mbox{\rm for} \quad |s|<\eta, \ |t| < \eta\, .
\end{equation}
Set  $\varepsilon := \min \{ \frac{\eta}{2 m} , \eta ,\mu \}$,
and let $\delta_1>0$ be chosen so that the inequality
in (\ref{ineq: c1}) holds for $(x,y)  \in \mathcal{B}(\overline{0},\delta_1)$.
Note that from (\ref{ineq: c1}) we have
\begin{equation}
\label{ineq: c11}
|\hat{T}(x,y)^{(2)}-\mu\,y | < \varepsilon \,y \quad \mbox{\rm and} \quad
|\hat{T}(x,y)^{(1)}-\hat{T}(x,0)^{(1)}| < \varepsilon \, | \, y \, |
\quad \mbox{\rm for} \quad
(x,y)  \in \mathcal{B}(\overline{0},\delta_1)\, .
\end{equation}
In this case there exist functions $\phi(x,y)$ and $\psi(x,y)$ of $(x,y) \in \mathcal{B}(0,\delta_1)$ such that
\begin{equation}
\label{ineq: phi psi varepsilon}
\left\{
\begin{array}{l}
|\phi(x,y)| < \varepsilon \, , \quad   |\psi(x,y) | < \varepsilon \, , \\ \\
\hat{T}(x,y)^{(2)} = \mu\, y + \phi(x,y)\, y \\ \\
\hat{T}(x,y)^{(1)} = \hat{T}(x,0)^{(1)}+ \psi(x,y)\, y
\end{array}
\right.
\qquad \mbox{\rm for } \ (x,y) \in \mathcal{B}(\overline{0},\delta_1)\, .
\end{equation}
From the Taylor expansion of $\hat{T}(x,y)$ about $(0,0)$ one can see that there exist  $c>0$ and $\delta_2>0$
such that
\begin{equation}
\label{ineq: c2}
\| \hat{T}(x,0)-\lambda\,x\,(0,1) \| < c\, x^2
\quad \mbox{\rm for} \quad |x| < \delta_2\, .
\end{equation}
Since $(\lambda\,x\,(0,1))^{(1)} = \lambda\,x$, we have from   (\ref{ineq: c2})   that
\begin{equation}
\label{ineq: c22}
|\, \hat{T}(x,0)^{(1)}-\lambda\,x  \, | < c\, x^2
\quad \mbox{\rm for} \quad |x| < \delta_2\, .
\end{equation}
In this case there exists a function $\xi(x)$ of $x \in (-\delta_2,\delta_2)$ such that
$|\xi(x) | < c$ for $|x|<\delta_2$, and
\begin{equation}
\hat{T}(x,0)^{(1)} = \lambda\,x+\xi(x) \, x^2 \quad \mbox{for } \quad |x| < \delta_2\, .
\end{equation}
Set $\delta_3:=\min(\delta_1,\delta_2)$. Then,
\begin{equation}
\hat{T}(x,y) = (\lambda\,x+\xi(x)\, x^2+\psi(x,y)\,y, \mu\,y+\phi(x,y)\,y)\, ,
\quad (x,y) \in \mathcal{B}(\overline{0},\delta_3)\, ,
\end{equation}
and consequently,
\begin{equation}
\label{eq: theta}
\tan \theta(\hat{T}(x,y)) = \frac{\mu\, y +\phi(x,y)\, y}{\lambda\,x+\xi(x) \,x^2+\psi(x,y)\,y}
=\frac{y}{x} \left(
\frac{\mu+\phi(x,y)}{\lambda + \xi(x)\,x+\psi(x,y)\frac{y}{x}}
\right)\, , \quad (x,y) \in \mathcal{B}(\overline{0},\delta_3)\, .
\end{equation}
Set  $\delta := \min\{\frac{\eta}{2 c},\delta_3\}$. Then
\begin{equation}
\label{eq:delta3}
\left| \xi(x)\, x \right| < c \cdot \frac{\eta}{2 c} = \frac{\eta}{2} \quad \mbox{\rm for } \quad |x|<\delta\, .
\end{equation}
By relation (\ref{ineq: phi psi varepsilon}) we have
\begin{equation}
\label{eq: psi}
\left| \psi(x,y) \frac{y}{x} \right| \leq \frac{\eta}{2 m} \cdot m = \frac{\eta}{2}
\quad \mbox{\rm for} \quad (x,y) \in \mathcal{B}(\overline{0},\delta)\cap \Omega(m)\, ,
\end{equation}
and
\begin{equation}
\label{eq: phi}
\left| \phi(x,y) \right| < \eta \quad \mbox{\rm for } \quad (x,y) \in \mathcal{B}(\overline{0},\delta)\, .
\end{equation}
By setting $s=\phi(x,y)$, $t=\xi(x) \,x+\psi(x,y)\,\frac{y}{x}$ for $(x,y)\in \mathcal{B}(\overline{0},\delta)$,
relations (\ref{eq: eta}) and (\ref{eq: theta}) through (\ref{eq: phi}) yield
\begin{equation}
\label{ineq: q}
\tan \theta (\hat{T}(x,y))
= \frac{y}{x}\, g(s,t) = \tan(\theta(x,y)) \, g(s,t)
> q \, \tan (\theta(x,y)) \quad \mbox{for} \quad   (x,y) \in \mathcal{B}(\overline{0},\delta)\cap \Omega(m).
\end{equation}
We now verify $\hat{T}(x,y) \in Q_1(\overline{0})$.
Since $\varepsilon < \mu$ by our choice of $\varepsilon$,
then from (\ref{ineq: c11}) we have $\hat{T}(x,y)^{(2)} > 0$.
Also, since $| \xi(x) \, x + \psi(x,y)\,\frac{y}{x} | < \eta < \lambda$, we have
\begin{equation}
\label{ineq: lambda eta}
\hat{T}(x,y)^{(1)} = x \,( \, \lambda + \xi(x) \, x + \psi(x,y)\,\frac{y}{x} \, )
> x \,( \, \lambda - \eta \, )
> 0 \quad \mbox{for} \quad
(x,y) \in \mathcal{B}(\overline{0},\delta)\cap \Omega(m).
\end{equation}
Relations (\ref{ineq: q}) and (\ref{ineq: lambda eta}) give (\ref{ineq: lemma q}),
thus completing the proof of the claim.
\hfill $\Box$
\bigskip

To complete the proof of the theorem, note that given $m>0$ arbitrary
and $\epsilon>0$
there exists $n_0$ such that
$\hat{T}^{n_0}(\hat{x}_0,\hat{y}_0)\in \mathcal{B}(0,\epsilon)\cap \mathcal{Q}_1(\overline{0})$
and $\hat{T}^{n_0}(\hat{x}_0,\hat{y}_0)\not \in \Omega(m)$.
Indeed, if this were not the case, then by Claim 3 and since
$\hat{T}^{n}(\hat{x}_0,\hat{y}_0)\rightarrow \overline{0}$,
one may choose $q>1$
such  (\ref{ineq: lemma q}) holds for
$(x,y) = \hat{T}^{n}(\hat{x}_0,\hat{y}_0)$ for all $n$ large enough, say $n\geq k$.
But then $m \geq \tan \theta(\hat{T}^{k+\ell}(\hat{x}_0,\hat{y}_0))  > q^\ell \,
\tan \theta (\hat{T}^k(\hat{x}_0,\hat{y}_0))$ for $\ell=1,2,\ldots$, which is impossible.
Thus we have shown that the sequence $\{\hat{T}^{n}(\hat{x}_0,\hat{y}_0)\}$
converges to $\overline{0}$ in such a way that it has a subsequence of points
outside the set $\Omega(m)$, for all $m>0$.
Since $\Theta$ maps the vertical positive semiaxis to a curve that emanates from the
origin with angle $\alpha \in (0,\frac{\pi}{2})$, it follows that there exists $n$
such that $T^n(x_0,y_0) \in \mathcal{Q}_2(\overline{0})$.
\end{proof}
\bigskip

\noindent {\bf Proof of Theorem \ref{th: saddle point}}.
The map $T$ is  a diffeomorphism on  a neighborhood $\mathcal{U}$
of $\overline{\rm x}$ onto its image.
The Unstable Manifold Theorem (page 282 in \cite{Katok})
guarantees the existence of the local unstable set
$\mathcal{W}^u_{\rm loc}(\overline{\rm x})$,
which is an invariant curve
that is tangential to ${\rm v}^\mu$ at $\overline{\rm x}$
and such that
\begin{equation}
\label{eq: expands}
\quad {\rm x} \in
\mathcal{W}^u_{\rm loc}(\overline{\rm x})
\quad \implies \quad
T^{-n}({\rm x}) \in
\mathcal{W}^u_{\rm loc}(\overline{\rm x})
\quad \mbox{\rm for }\  n \in \mathbb{N}
\quad \mbox{\rm and}\quad
T^{-n}({\rm x}) \rightarrow \overline{\rm x}
\end{equation}
Since $\mu>1$ and the eigenspace $E^\mu$ is not a coordinate axis,
the entries of eigenvectors ${\rm v}^\mu$ are nonzero.
Further, since $T$ is competitive, the entries of ${\rm v}^\mu$ have
different sign.  Thus points in $\mathcal{W}^u_{\rm loc}(\overline{\rm x})$
that are close enough to $\overline{\rm x}$ are comparable.
It follows from this and from  (\ref{eq: expands})
that $\mathcal{W}^u_{\rm loc}(\overline{\rm x}) $ is
linearly ordered by $\preceq_{se}$.
We claim that points ${\rm x} \in \mathcal{Q}_2(\overline{\rm x})
\cap \mathcal{W}^u_{\rm loc}(\overline{\rm x})$ are subsolutions.
Indeed, the set $ \mathcal{Q}_2(\overline{\rm x})
\cap \mathcal{W}^u_{\rm loc}(\overline{\rm x})$
is invariant.
If ${\rm x} \in \mathcal{Q}_2(\overline{\rm x})
\cap \mathcal{W}^u_{\rm loc}(\overline{\rm x})$ had
${\rm x} \preceq_{se} T({\rm x}) \preceq \overline{\rm x}$, then
$T^{-n}({\rm x}) \preceq_{se}  {\rm x} \preceq \overline{\rm x} $
for $n\in \mathbb{N}$, which contradicts (\ref{eq: expands}).
Similarly, points ${\rm x} \in \mathcal{Q}_4(\overline{\rm x})
\cap \mathcal{W}^u_{\rm loc}(\overline{\rm x})$ are supersolutions.
We have, $\mathcal{W}^u(\overline{\rm x}) = \bigcup_{n=0}^\infty
T^n (  \mathcal{W}^u_{\rm loc}(\overline{\rm x}) )$,
hence $\mathcal{W}^u(\overline{\rm x})$
is a nested union of connected and linearly ordered sets,
hence itself is connected and linearly ordered.
Thus $\mathcal{W}^u(\overline{\rm x})$
is the graph of a decreasing function of the first coordinate.
Let ${\rm y}$ and ${\rm z}$ be the endpoints of
$\mathcal{W}^u(\overline{\rm x})$, with
${\rm y} \preceq {\rm z}$.
Suppose ${\rm y} \in {\rm int}\, \mathcal{R}$.
Since ${\rm y}$ is an accumulation point of subsolutions
in $\mathcal{W}^u(\overline{\rm x})$,
${\rm y}$ is also a subsolution
by continuity of $T$,  and it follows that
${\rm y}$ is   a fixed point.
Similar reasoning applies to ${\rm z}$.

To see that $\mathcal{W}^s(\overline{\rm x}) = \mathcal{C}$, note that
by part (B) of Theorem \ref{th: basins},
iterates of points ${\rm x} \in
\mathcal{R}\setminus\mathcal{C}$
eventually enter
${\rm int} (\mathcal{Q}_2(\overline{\rm x})\cup \mathcal{Q}_4(\overline{\rm x}))\cap \mathcal{R}$.
If for some $n$, $T^n({\rm x}) \in {\rm int}\, \mathcal{Q}_2(\overline{\rm x})$ (say), then
there exists ${\rm y} \in \mathcal{W}^u(\overline{\rm x})$
such that $T^n({\rm x}) \preceq_{se} {\rm y}$, and consequently,
$T^{n+k}({\rm x}) \preceq_{se}T^k( {\rm y})$ for $k \in \mathbb{N}$.
Since ${\rm y}$ is a subsolution,  it follows that $T^\ell ({\rm x}) \not \rightarrow \overline{\rm x}$.
Similarly, $T^n({\rm x}) \in {\rm int}\, \mathcal{Q}_4(\overline{\rm x})$ implies
$T^\ell ({\rm x}) \not \rightarrow \overline{\rm x}$.
Thus $T^\ell({\rm x}) \rightarrow \overline{\rm x}$ if and only if ${\rm x} \in \mathcal{C}$.
\hfill $\Box$

\bigskip

\noindent {\bf Proof of Theorem \ref{th: 2 and 4 all cases}}\ \
Since $\overline{{\rm x}}$ is a fixed point of $T$ and ${\rm v}$ is an eigenvector with associated
eigenvalue $\mu$, we have
\begin{equation}
\label{eq: taylor 3.1}
T(\overline{{\rm x}}+t\,{\rm v}) = \overline{{\rm x}} + t\,\mu\, {\rm v} + o(t).
\end{equation}
Since ${\rm v}^{(1)} {\rm v}^{(2)} <0$, we may assume without loss of generality
that  ${\rm v} \preceq_{se} \overline{0}$.
To prove (i) assume
$\mu >1$, so in particular $(\mu-1) {\rm v} \preceq_{se} \overline{0}$.
If ${\rm int}\,Q_2(\overline{{\rm x}})\cap \mathcal{R}  = \emptyset$, set $t_0=0$, and if
${\rm int}\,Q_2(\overline{{\rm x}})\cap \mathcal{R}  \neq \emptyset$, choose $t_0 >0$ such that
$\llbracket \overline{{\rm x}}+t_0\,{\rm v},\overline{{\rm x}} \rrbracket$ is a subset of $ \mathcal{R}$ and has
no fixed points other than $\overline{{\rm x}}$, and
 \begin{equation}
\label{eq: taylor 3.1 2}
\displaystyle
\frac{1}{t}\left(\,T(\overline{{\rm x}}+t\,{\rm v}) - ( \overline{{\rm x}} + t\,{\rm v}) \,\right) =
(\mu-1)\, {\rm v} + \left(\frac{o_1(t)}{t},\frac{o_2(t)}{t}\right)  \preceq_{se} \overline{0}\, , \quad t \in (0,t_0]\, .
\end{equation}
Hence $T^(\overline{{\rm x}}+t \,{\rm v}) \preceq_{se} \overline{{\rm x}}+t\,{\rm v}$ for $t \in (0,t_0]$.
Similarly, if $ \mathcal{R} \cap {\rm int}\, Q_4(\overline{{\rm x}}) = \emptyset$ set $t_1=0$, and if
$ \mathcal{R} \cap {\rm int}\, Q_4(\overline{{\rm x}}) \neq \emptyset$ choose $t_1 < 0$
such that
$\overline{{\rm x}}+t\,{\rm v} \preceq_{se} T^(\overline{{\rm x}}+t \,{\rm v}) $ for $t \in [t_1,0)$.
Note that $t_1$ and $t_2$ are not both $0$.
Define $I := \llbracket \overline{{\rm x}}+t_0\,{\rm v},\overline{{\rm x}}+t_1\,{\rm v}\rrbracket$.
The set $I$ is a relative neighborhood of $\overline{{\rm x}}$ in $ \mathcal{R}$,
and  $I \cap {\rm int}\,(Q_2(\overline{{\rm x}})\cup Q_4(\overline{{\rm x}}))$ has no fixed points.
 Then for every relative neighborhood $U \subset I$ of $\overline{{\rm x}}$,
$U\cap \,{\rm int}\,Q_2(\overline{{\rm x}})$ contains a subsolution and
 $U\cap \,{\rm int}\,Q_4(\overline{{\rm x}})$ contains a supersolution.
For ${\rm x} \in I\cap {\rm int}\,(Q_2(\overline{{\rm x}})\cup Q_4(\overline{{\rm x}}))$,
there exists a  ${\rm y} \in I$ which is not a fixed point of $T$ such that
such that   ${\rm y}$ is a subsolution and ${\rm x} \preceq_{se} {\rm y} \preceq_{se} \overline{{\rm x}}$ or
${\rm y}$ is a supersolution and $\overline{{\rm x}} \preceq_{se} {\rm y} \preceq_{se} {\rm x}$.
Then either $T^n({\rm x}) \preceq_{se} T^n({\rm y})$ or $T^n({\rm y}) \preceq_{se}{\rm x} \preceq_{se} {\rm y} \preceq \overline{{\rm x}}$.
Since there exists $n_0$ such that $T^{n}({\rm y}) \not \in I$ for $n \geq n_0$,
we have $T^{n}({\rm x}) \not \in I$ for $n \geq n_0$.
The proof of (ii) is similar and we skip it.
\hfill $\Box$
\bigskip

\noindent {\bf Proof of Theorem \ref{th: 2 and 4 all cases part 2}}\ \
We prove statement (i.) only, as the proof of statements (ii.)--(iv.) is similar.
Assume $\ell$ is odd and $(c_\ell,d_\ell) \preceq_{se} \overline{0}$.
If $c_\ell\,d_\ell < 0$, then there exists $t_*>0$ such that
 \begin{equation}
\label{eq: taylor 4.1 2}
\displaystyle
\frac{1}{t^\ell}\left(\,T(\overline{\rm x}+t\,{\rm v}) - ( \overline{\rm x} + t\,{\rm v}) \,\right) =   (c_\ell\,,\,d_\ell) + (O_1(t),O_2(t))
 \preceq_{se} \overline{0}\, , \quad t \in (-t_*,t_*) ,
\end{equation}
which implies $T(\overline{\rm x}+t\,{\rm v})  \preceq_{se} \overline{\rm x} + t\,{\rm v}$ for $t \in (0,t_*)$ and
$\overline{\rm x}+t\,{\rm v}  \preceq_{se}T( \overline{\rm x} + t\,{\rm x})$ for $t \in (-t_*,0)$.
The rest of the proof now proceeds as in the proof of Theorem \ref{th: 2 and 4 all cases}.
If $c_\ell\neq 0$ and $T(\overline{\rm x}+t\, {\rm x})^{(2)}$ is affine, then
$d_\ell=0$ and $O_2(t) = 0$ in (\ref{eq: taylor 4.1 2}), and the proof proceeds as before.
The same reasoning applies to the case
where $d_\ell\neq 0$ and $T(\overline{\rm x}+t v)^{(1)}$ is affine.
\hfill $\Box$



\begin{thebibliography}{77}
\bibitem{ASY} K. T. Alligood, T. Sauer, J. Yorke,
Chaos: An Introduction to Dynamical Systems,
1996 Springer Verlag, New York.

\bibitem{AGGL} A. M. Amleh, D. A. Georgiou, E. A. Grove, and G. Ladas, On the Recursive Sequence
$y_{n+1} = \alpha + \frac{x_{n-1}}{x_n }$, {\it J. Math. Anal. Appl.} 233(1999), 790-798.

\bibitem{BM} S. Basu and O. Merino, On the global behavior of solutions to a planar system of difference equations,  {\it Comm. Appl. Nonlinear Anal.}  16(2009),  no. 1, 89--111.


\bibitem{BKK1} D\v{z}. Burgi\,{c}, S. Kalabu\v{s}i\'{c} and M. R. S. Kulenovi\'{c}, Nonhyperbolic Dynamics for Competitive Systems in the Plane and Global Period-doubling Bifurcations, {\it Adv. Dyn. Syst. Appl.} 3(2008), (to appear).


\bibitem{CaKuLaMe} E. Camouzis, M. R. S. Kulenovi\'{c}, G. Ladas, and O. Merino, Rational Systems in the Plane,  \emph{J. Differ. Equ. Appl.} 15 (2009), 303-323.


\bibitem{CK} D. Clark and M. R. S. Kulenovi\'c, On a Coupled System
of Rational Difference Equations, {\it Comput. Math. Appl.} 43
(2002),  849-867.

\bibitem{CKS} D. Clark, M. R. S. Kulenovi\'c, and J.F. Selgrade,
Global Asymptotic Behavior of a Two Dimensional Difference Equation
Modelling Competition, {\it Nonlinear Anal., TMA}  52(2003),
1765-1776.

\bibitem{CAKS} C. A. Clark, M. R. S. Kulenovi\'c, and J.F. Selgrade,
On a system of rational  difference equations, \emph{J. Differ. Equ.
Appl.} 11(2005),  565-580.

\bibitem{CLCH} J. M. Cushing, S. Levarge, N. Chitnis and S. M.
Henson, Some discrete competition models and the competitive
exclusion principle , \emph{J. Differ. Equ.  Appl.} 10(2004),
1139-1152.

\bibitem{Dancer and Hess}
E. Dancer and P. Hess,
Stability of fixed points for order preserving discrete-time
dynamical systems, {\it J. Reine Angew Math.}{ \bf 419},
125--139, 1991.

\bibitem{GDKN1}  M. Gari\'{c}-Demirovi\'{c}, M. R. S. Kulenovi\'{c} and M. Nurkanovi\'{c}, Global Behavior of Some Competitive Systems of Rational Difference Equations in the Plane, (to appear).

\bibitem{deMS} P. de Mottoni and A. Schiaffino,
Competition systems with periodic coefficients:
A geometric approach, {\it J. Math. Biol.} {\bf 11} (1981), 319--335.

\bibitem{FY1} J. E. Franke and A-A. Yakubu, Mutual exclusion versus coexistence for discrete competitive systems, {\it J. Math. Biol.} 30(1991), 161--168.

\bibitem{FY2} J. E. Franke and A-A. Yakubu, Global attractors in competitive systems, {\it Nonlin. Anal. TMA} 16(1991), 111--129.

\bibitem{FY3} J. E. Franke and A-A. Yakubu, Geometry of exclusion principles in discrete systems, {\it J. Math. Anal. Appl.} 168(1992), 385--400.

\bibitem{Hartman} P. Hartman, Ordinary Differential Equations,
J. Wiley \and Sons, Inc., New York, 1964.

\bibitem{HP} P. Hess and  P. Pola\v{c}ik, Boundedness of prime periods of stable cycles and convergence to fixed points in discrete monotone dynamical systems, {\it SIAM J. Math. Anal.}, 24 (1993), 1312--1330.

\bibitem{HS1} M. Hirsch and H. Smith, Monotone Dynamical Systems,  Handbook of Differential Equations,
Ordinary Differential Equations (second volume),  239-357,
Elsevier B. V., Amsterdam, 2005.

\bibitem{Katok} B. Hasselblatt and A. Katok,
A First Course in Dynamics with a Panorama of Recent Devlopments,
Cambridge University Press, New York, 2003.


\bibitem{KUL} M. R. S. Kulenovi\'{c} and G. Ladas, \textit{Dynamics of Second Order Rational Difference Equations,} Chapman \& Hall/CRC, Boca Raton, London, 2001.

\bibitem{KUM}  M. R. S. Kulenovi\'{c} and O. Merino, {\it Discrete Dynamical Systems and Difference Equations
 with Mathematica}, Chapman\& Hall/CRC Press, Boca Raton, 2002.

\bibitem{KUM1}  M. R. S. Kulenovi\'{c} and O. Merino, Competitive-Exclusion versus Competitive-Coexistence for Systems in the Plane, \emph{Discrete Contin. Dyn. Syst. Ser. B} 6(2006), 1141-1156.

\bibitem{KUM2}  M. R. S. Kulenovi\'{c} and O. Merino, Global Bifurcations
for Competitive Systems in the Plane, \emph{Discrete Contin. Dyn. Syst. Ser. B} 9(2009), (to appear).

\bibitem{KN4}  M. R. S. Kulenovi\'{c} and M. Nurkanovi\'{c}, Asymptotic behavior of a competitive system of linear fractional difference equations.  {\it Adv. Difference Equ.}  2006, Art. ID 19756, 13 pp.
    
    
\bibitem{PT1}     P. Pola\v{c}ik and I. Tere\v{s}\v{c}cak, Convergence to cycles as a typical asymptotic behavior in smooth strongly monotone discrete-time dynamical systems, {\it Arch. Rational Mech. Anal.} 116 (1992),  339--360. 
    

\bibitem{Robinson} C. Robinson, Dynamical Systems, Second Edition,
CRC Press, Boca Raton, 1999.

\bibitem{SZ} J. F. Selgrade and M. Ziehe, Convergence to equilibrium
in a genetic model with differential viability between the sexes,
{\it J. Math. Biol.} {\bf 25} (1987), 477-490.

\bibitem{S1986JDE} H. L. Smith, Periodic competitive differential equations and the discrete dynamics of competitive maps, \emph{J.
Differential Equations} 64 (1986), 165-194.

\bibitem{S1986SIAM} H. L. Smith, Invariant curves for mappings,
\emph{SIAM J. Math. Anal.} 17 (1986), 1053-1067.

\bibitem{S2-1986SIAM} H. L. Smith, Periodic solutions of periodic competitive and cooperative systems,
{\it SIAM J. Math. Anal.} {\bf 17}(1986), 1289--1318.

\bibitem{Sm1}  H. L. Smith, Planar Competitive and Cooperative Difference Equations,\emph{J. Differ. Equ.  Appl. } 3(1998), 335-357.



\end{thebibliography}
\end{document}